\documentclass[11pt,letterpaper]{amsart}
\pdfoutput=1
\usepackage{geometry}                
\usepackage{graphicx}
\usepackage{amsfonts,amsthm,amsmath,amssymb,latexsym, amscd, euscript}
\usepackage{graphicx,color}
\usepackage[all]{xy}
\usepackage{epsfig}
\usepackage{labelfig}
\usepackage{mathrsfs}

\usepackage{amssymb}
\usepackage{epstopdf}
\theoremstyle{plain}
\newtheorem{theorem}{Theorem}[section]
\newtheorem{corollary}[theorem]{Corollary}

\newtheorem{lemma}[theorem]{Lemma}
\newtheorem{remark}[theorem]{Remark}

\newtheorem{definition}[theorem]{Definition}

\begin{document}

\theoremstyle{definition} 

\newtheorem*{notation}{Notation}  

\theoremstyle{plain}      

\def\H{{\mathbb H}}
\def\F{{\mathcal F}}
\def\R{{\mathbb R}}
\def\Q{\hat{\mathbb Q}}
\def\Z{{\mathbb Z}}
\def\E{{\mathcal E}}
\def\N{{\mathbb N}}
\def\X{{\mathcal X}}
\def\Y{{\mathcal Y}}
\def\C{{\mathbb C}}
\def\D{{\mathbb D}}
\def\G{{\mathcal G}}
\def\T{{\mathcal T}}

\title{Ergodicity of the geodesic flow on symmetric surfaces}


\date{}

\author{Michael Pandazis and Dragomir \v Sari\'c}

\address[Michael Pandazis]{PhD Program in Mathematics, The Graduate Center, CUNY \\ 365 Fifth Ave., N.Y., N.Y., 10016, USA.}
\email{mpandazis@gradcenter.cuny.edu}

 \thanks{The second author was partially supported by the Simons Foundation Collaboration Grant 346391 and by PSCCUNY grants.}

\address[Dragomir \v Sari\' c]{PhD Program in Mathematics, The Graduate Center, CUNY \\ 365 Fifth Ave., N.Y., N.Y., 10016 and\newline Department of Mathematics, Queens College, CUNY\\ 65--30 Kissena Blvd., Flushing, NY 11367, USA.}
\email{Dragomir.Saric@qc.cuny.edu}

\maketitle

\begin{abstract}
We consider conditions on the Fenchel-Nielsen parameters of a Riemann surface $X$ that guarantee the surface $X$ is of parabolic type. An interesting class of Riemann surfaces for this problem is the one with finitely many topological ends. In this case the length part of the Fenchel-Nielsen coordinates can go to infinity for {parabolic $X$}.  When the surface $X$ is end symmetric, we prove that {$X$ being parabolic} is equivalent to the covering group being of the first kind. Then we give necessary and sufficient conditions on the Fenchel-Nielsen coordinates of a half-twist symmetric surface $X$ such that {$X$ is parabolic}. As an application, we solve an open question from the prior work of Basmajian, Hakobyan and the second author.
\end{abstract}

\section{Introduction}

A Riemann surface $X$ is of {\it parabolic type} if it admits no Green's function (see Ahlfors-Sario \cite{AhlforsSario}). 
A Green's function on $X$ is a harmonic function $u$ with the logarithmic singularity at a single point of $X$ such that $\lim_{z\to\partial X}u(z)=0$ (\cite{AhlforsSario}). The class of parabolic Riemann surfaces is denoted by $O_G$ and the reader should not be confused with the fact that most of these surfaces support a hyperbolic metric.

The function theoretic property of {$X$ being parabolic} has deep connections with other natural properties of Riemann surfaces. To name a few, {$X$ being parabolic} is equivalent to \cite{Nevanlinna:criterion,AhlforsSario,Agard,Tukia,Nicholls1,Sullivan,Astala-Zinsmeister,Bishop,Fernandez-Melian,Saric}:
\begin{itemize} 
\item the geodesic flow on the unit tangent bundle $T^1(X)$ of $X$ is ergodic,
\item the boundary at infinity has zero harmonic measure,
\item the Poincar\'e series $\sum_{\gamma \in\Gamma} e^{-d(z,\gamma (z))}$ is divergent, where $X=\mathbb{H}/\Gamma$, $d(\cdot ,\cdot )$ is the hyperbolic distance in $\mathbb{H}$ and $z\in\mathbb{H}$,
\item the Brownian motion on $X$ is recurrent,
\item the limit set of a quasiconformal deformation of $\Gamma$ has {Bowen's} property, and
\item almost every horizontal trajectory of every finite area holomorphic quadratic differential on $X$ is recurrent.
\end{itemize}

The type problem for Riemann surfaces is a question of determining whether a Riemann surface given by an explicit construction which usually depends on some countable family of parameters is {parabolic}. This problem has been extensively studied by many authors when the Riemann surfaces were naturally defined by either gluing construction along the slits or other constructions motivated by complex analysis considerations (for example, see Ahlfors-Sario \cite{AhlforsSario} and Milnor \cite{Milnor}). Determining whether {$X$ is parabolic} can be a challenging problem for a specific construction.

Basmajian, Hakobyan and  the second author \cite{BHS} used the Fenchel-Nielsen coordinates to determine the type of a Riemann surface. There is a dichotomy in the flavor of the results depending on the number of topological ends of the surface $X$. A Riemann surface $X_C$ with a Cantor set of ends admits no Green's function if each boundary geodesic at the level $n$ has length bounded above by $\frac{n}{2^n}$(see \cite[Theorem 10.3]{BHS}). Due to the presence of a large space of ends, McMullen \cite{McMullen} proved that if all boundary geodesics of the pants decomposition of $X_C$ have lengths between two positive constants then {$X_C$ is not parabolic}. The second author \cite[Theorem 7.4]{Saric} proved that {$X_C$ is not parabolic} when the lengths of geodesics at the level $n$ are at least $\frac{n^r}{2^n}$ for $r>2$. Notice that {$X_C$ is parabolic} for small $\ell_n$ {independent} of the {choice} of twists because twisting along geodesics with lengths bounded above is a quasiconformal deformation which preserves parabolicity.

A Riemann surface $X_f$ with countably many punctures that accumulate to a single topological end is called a {\it flute surface} (see Basmajian \cite{Basmajian}). We denote by $\{\alpha_n\}_{n=1}^{\infty}$ the geodesic boundaries of a fixed pants decomposition of $X_f$ (see Figure 2). Let $\ell_n$ be the length of $\alpha_n$ and let $t_n$ be the twist on $\alpha_n$. A flute surface is called a {\it zero-twist flute surface} if $t_n=0$ for all $n$. 
Basmajian, Hakobyan and the second author \cite{BHS} proved the following for a zero-twist flute surface $X_f^0=\mathbb{H}/\Gamma$:
the covering Fuchsian group $\Gamma$ is of the first kind if and only if  $\sum_{n=1}^{\infty}e^{-\frac{\ell_n}{2}} =\infty$ if and only if {$X_f^{0}$ is parabolic}.

Our first result is a necessary and sufficient condition for symmetric surfaces with small number of ends to be {parabolic}. A surface is {\it symmetric} if it can be decomposed into two parts by the union of simple closed geodesics and bi-infinite geodesics that are interchanged by an orientation reversing isometry. A flute surface is symmetric if twists $t_n$ belong to the set $\{0,\frac{1}{2}\}$; in particular, zero-twist ($t_n=0$ for all $n$)  and half-twist ($t_n=1/2$ for all $n$) flute surfaces are symmetric. We prove (see Theorem \ref{thm:equiv-par-com})

\begin{theorem}
Let $X_f=\mathbb{H}/\Gamma$ be a flute surface with $t_n\in\{ 0,\frac{1}{2}\}$ for all $n$. Then {$X_f$ is parabolic} if and only if $\Gamma$ is of the first kind.
\end{theorem}

Given a compact exhaustion $\{X_k\}$ of a Riemann surface $X$, an end of the surface is represented by a nested sequence $\{C_K\}$ of open subsets of $X$, where $C_k$ is a connected component of $X - X_k$. We say such an end is {\it accumulated by genus} if each $C_k$ corresponding to the end has positive genus. A Riemann surface $X$ with finitely many ends accumulated by genus is {\it end symmetric} if each end surface, which is a bordered surface with one closed geodesic  on its boundary and one topological end, is symmetric in the above sense (see Figures 5 and 6). For each handle in an end surface we choose a simple closed geodesic $\beta_n$ that cuts off the handle from the end surface.  In the torus complement of $\beta_n$, we choose a simple closed geodesic $\gamma_n$. We assume that the lengths $\ell (\beta_n)$ and $\ell (\gamma_n)$ are between two positive constants. 
We prove that $X=\mathbb{H}/\Gamma$ {being parabolic} is equivalent to $\Gamma$ being of the first kind for an end symmetric Riemann surface $X$ (see Theorem \ref{thm:symmetric-finite-ends}). 

\begin{theorem}
Let $X=\mathbb{H}/\Gamma$ be a Riemann surface that is end symmetric with the lengths $\ell (\beta_n)$ and $\ell (\gamma_n)$ between two positive constants. Then {$X$ is parabolic} if and only if $\Gamma$ is of the first kind.
\end{theorem}

The above two theorems reduce the question whether $X=\mathbb{H}/\Gamma$ is parabolic to the question whether $\Gamma$ is of the first kind for {a flute surface $X$ with only zero or half twists and for} an end symmetric Riemann surface $X$.

A flute surface $X_f^{1/2}$ is called a {\it half-twist flute surface} if $t_n=\frac{1}{2}$ for all $n$.
Basmajian, Hakobyan and the second author (see \cite[Section 9.1]{BHS}) used a semi-local argument (the countable sum of the moduli  of the union of two non-standard half-collars) to establish that 
\begin{equation}
\label{eq:half-twist-parab-sum}
\sum_{n=1}^{\infty}e^{-\frac{\ell_n}{4}}=\infty
\end{equation} implies that $X_f^{1/2}$ {is parabolic}. 
Under the additional conditions that $\ell_n$ is increasing and concave, the condition (\ref{eq:half-twist-parab-sum}) is equivalent to both $X_f^{1/2}$ {being parabolic} and $\Gamma$ {being} of the first kind (see \cite[Theorem 9.7]{BHS}).  For $\sigma_n=\ell_n-\ell_{n-1}+\cdots +(-1)^{n-1}l_1$, if
\begin{equation}
\label{eq:non-complete-1/2flute}
\sum_{n=1}^{\infty}e^{-\frac{\sigma_n}{2}}<\infty
\end{equation}
then $\Gamma$ is of the second kind and necessarily $X_f^{1/2}$ {is not parabolic} (see \cite[Theorem 9.6]{BHS}). 

To illustrate the difference between (\ref{eq:half-twist-parab-sum}) and (\ref{eq:non-complete-1/2flute}) when {the} $\ell_n$ are increasing but not concave, we define {the} {\it Hakobyan slice} (see \cite[Example 9.9]{BHS}) to consist of {a} two parameter family of half-twist flute surfaces $X_{a,b}$, with $a>0$ and $b>0$, and
$$
\ell_{2n}=a\ln (n+1)+b\ln n ,\ \ \ell_{2n+1}=(a+b)\ln (n+1).
$$
In the first quadrant of the $ab$-plane (Figure 1), the two regions $X_{a,b}=?$  satisfy {n}either property (\ref{eq:half-twist-parab-sum}) {n}or (\ref{eq:non-complete-1/2flute}). An open question was to decide the nature of the surfaces in the regions $X_{a,b}=?$, i.e. where $a+b >4$ and $min\{a,b\}\le2$ (see \cite[Question 1.9]{BHS}).

 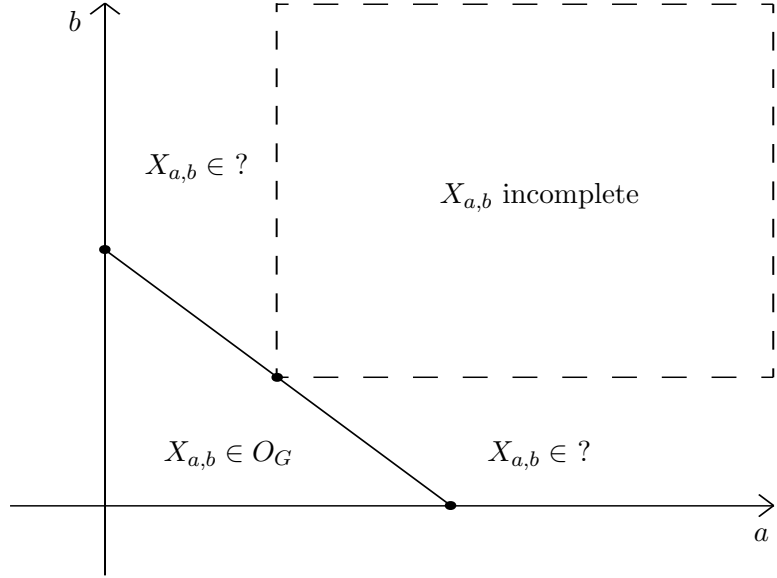
\begin{figure}[h]
\leavevmode \SetLabels
\L(.55*.65) $X_{a,b}$ incomplete\\
\L(.26*.2) $X_{a,b}\in O_G$\\
\L(.6*.2) $X_{a,b}\in \ ?$\\
\L(.24*.7) $X_{a,b}\in \ ?$\\
\L(.88*.06) $a$\\
\L(.16*.95) $b$\\
\endSetLabels
\begin{center}
\AffixLabels{\centerline{\epsfig{file= 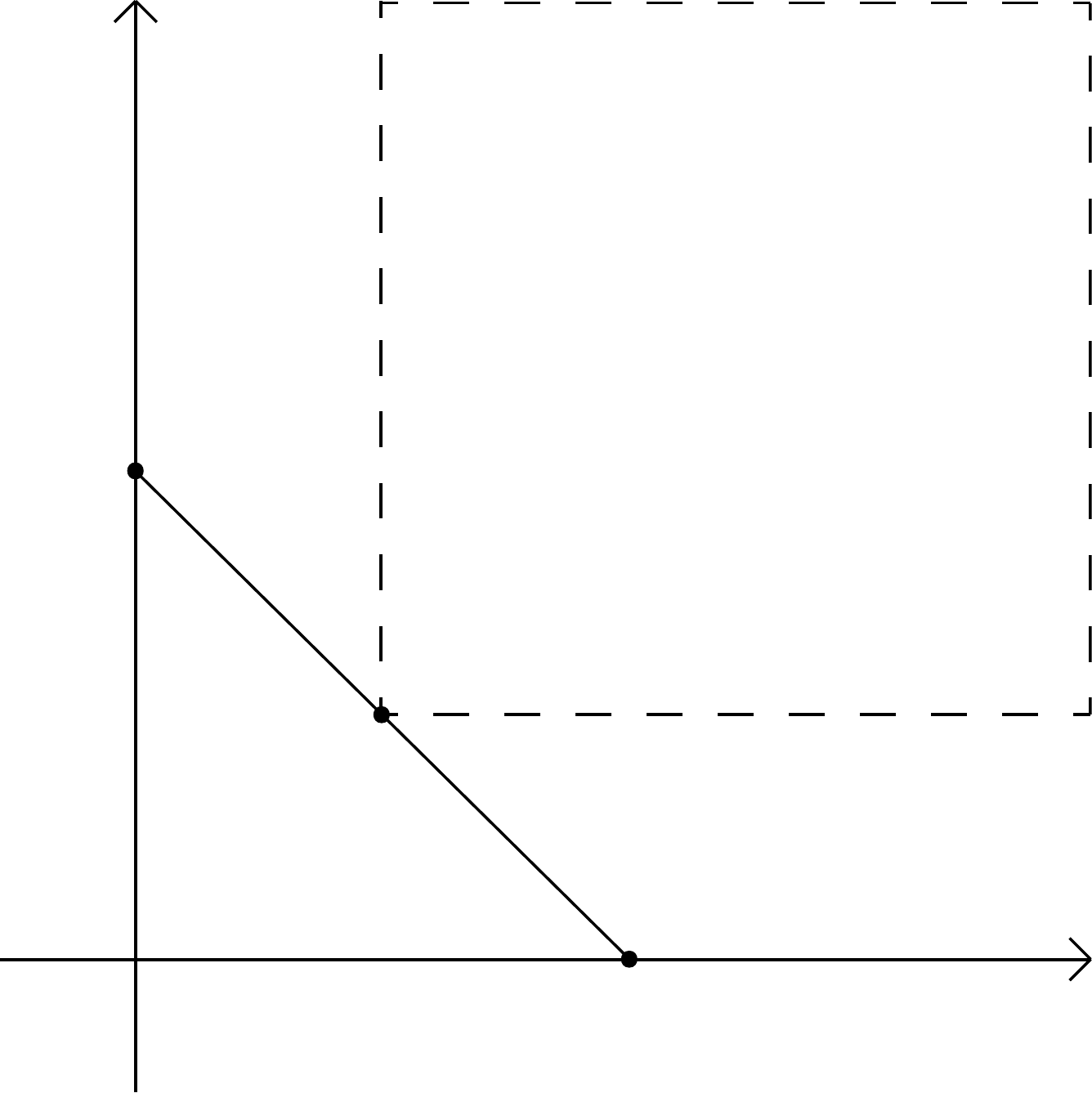,width=4in,height=3in,angle=0}}}
\vspace{-20pt}
\end{center}
\caption{A slice of flutes $X_{a,b}$.}
\end{figure}

We prove (see Theorem \ref{thm:half-twist-O_G})
\begin{theorem}
\label{thm:sum-flute-1/2}
Assume that  the lengths $\ell_n$ of the boundary geodesics of the pants decomposition of a half-twist flute surface $X_f^{1/2}=\mathbb{H}/\Gamma$ are increasing. Then the following are equivalent.
\begin{enumerate}
\item the covering group $\Gamma$ of $X_f^{1/2}$ is of the first kind,
\item $\sum_{n=1}^{\infty}e^{-\frac{\sigma_n}{2}}=\infty$, and
\item $X_f^{1/2}$ {is parabolic}.
\end{enumerate}
\end{theorem}

A direct corollary to the above theorem is that the half-twist surfaces in the two regions $X_{a,b}=?$ are parabolic which implies the existence of parabolic half-twist surfaces  for which $\sum_{n=1}^{\infty}e^{-\frac{\ell_n}{4}}<\infty$ (thus the non-concave lengths impose a different geometric condition for parabolic surfaces and the sufficient condition for parabolicity from \cite{BHS} is far from necessary).

\begin{corollary}
In the slice $X_{a,b}$ above, the Riemann surface $X_{a,b}$ has Fuchsian covering group of the second kind if and only if $\min\{ a,b\}>2$. Moreover, $X_{a,b}$ is parabolic if and only if $\min\{ a,b\}\leq 2$.
\end{corollary}

Let $X^{1/2}=\mathbb{H}/\Gamma$ be a Riemann surface with finitely many ends accumulated by genus such that in each end surface the twists around the boundary geodesics $\alpha_n$ are equal to $1/2$. Let $\beta_n$ be geodesics that cut off handles and let $\gamma_n$ be the closed geodesic in each torus complement of $\beta_n$. We assume that the lengths $\ell (\beta_n)$ and $\ell (\gamma_n)$ are between two positive constants. Then we have (see Theorem \ref{thm:symmetric-1/2-finite-ends})

\begin{theorem}
\label{thm:finite-ends-1/2}
Let $X^{1/2}$ be a Riemann surface with finitely many ends whose twists are $1/2$ in each end as above. Assume that the lengths $\ell (\beta_n)$ and $\ell (\gamma_n)$ are between two positive constants and let $\ell_n$ be the length of $\alpha_n$.
Then the following are equivalent.
\begin{enumerate}
\item the covering group $\Gamma$ of $X^{1/2}$ is of the first kind,
\item $\sum_{n=1}^{\infty}e^{-\frac{\sigma_n}{2}}=\infty$, and
\item $X^{1/2}$ is parabolic.
\end{enumerate}
\end{theorem}

The method of the proof of Theorem \ref{thm:sum-flute-1/2} could be of independent interest. The main direction is to prove that (2) implies (1). The goal is to prove that the nested sequence of lifts of $\alpha_n$ accumulates to a single point on the ideal boundary of $\mathbb{H}$.  This is challenging because the half-twists play the role of off-setting the lifts of the boundary geodesics $\alpha_n$ by the lengths $\frac{\ell_n}{2}$ which  converge to infinity. When we connect the consecutive lifts $\widetilde{\alpha_n}$ of $\alpha_n$ by common orthogeodesic arcs $\eta_n$, the foots on $\widetilde{\alpha_n}$ from the left and the right are a  distance $\frac{\ell_n}{2}$ apart. This problem can be compared to the part of the proof by Kahn and Markovic \cite{KahnMarkovic} of the Surface Subgroup conjecture where their offset is by the length $1$ in the same direction. 

Starting from a point on $\widetilde{\alpha_1}$, we can form a concatenation of the summits of Saccheri quadrilaterals with bases $\eta_n$ (see \cite[Figure 9.2]{BHS}). Then $e^{-\frac{\sigma_{n-1}}{2}}+e^{-\frac{\sigma_{n+1}}{2}}$ is asymptotically the length of the $n$-th summit. 
Showing directly that (2) implies that the lifts $\widetilde{\alpha_n}$ accumulate to a single point on the ideal boundary of $\mathbb{H}$ seems very prohibitive. Instead we add one more geodesic between each consecutive pair $\widetilde{\alpha_n}$ and $\widetilde{\alpha_{n+1}}$ to obtain a sequence of wedges which can be though of as  part of an ideal triangulation of $\mathbb{H}$ (the dotted geodesics in  Figure 10). Then we compute the shears (see \cite[Theorem C]{Saric2011}) along the nested sequence in order to show that the length of a piecewise horocyclic path along the wedges is infinite which shows that the accumulation is to a single point. For convenience, we prove all the necessary properties and give detailed analysis of the piecewise horocyclic path and shears in the Appendix.

\section{The class of parabolic Riemann surfaces}

A Riemann surface $X$ is said to be {\it infinite} if its fundamental group is infinitely generated. An infinite Riemann surface supports a unique hyperbolic metric in its conformal class and it is isometric to $\mathbb{H}/\Gamma$, where $\mathbb{H}$ is the hyperbolic plane and $\Gamma$ is a Fuchsian covering group.

Any infinite Riemann surface $X$ has a topological pants decomposition (see \cite{Ker}, \cite{Richards}). A {\it topological pair of pants} is a bordered surface homeomorphic to a sphere minus three open disks whose {boundaries} are Jordan curves. A {\it geodesic pair of pants} is a bordered hyperbolic surface whose interior is homeomorphic to the interior of a topological pair of pants and whose boundaries are either simple closed geodesics or punctures. We do not allow all three boundary components to be punctures. 

Fix a topological pants decomposition $\mathcal{P}$ of an infinite Riemann surface $X$.
By replacing each boundary curve of every topological pair of pants of $\mathcal{P}$ by either a homotopic simple closed geodesic or a puncture, we can straighten the topological pairs of pants into geodesic pairs of pants. The union of such obtained geodesic pairs of pants covers the convex core $\mathcal{C}(X)$ of $X$ with the exception of open geodesics on its boundary (see \cite{BasmajianSaric}). In fact, the Riemann surface $X$ is obtained by attaching hyperbolic funnels to closed boundary geodesics of the convex core and by attaching geodesic half-planes to open boundary geodesics of the convex core (see \cite{BasmajianSaric} and also \cite{AlvarezRodriguez}). A similar fact holds for any orientable complete Riemannian surface, without any restriction on curvature, see \cite{Portilla}. When $\Gamma$ is of the first kind, then each topological pants decomposition $\mathcal{P}$ straightens necessarily to a geodesic pants decomposition $\hat{\mathcal{P}}$ of the whole surface $X$ (see \cite{BasmajianSaric}). 

Given a geodesic pants decomposition, denote by $\{\alpha_n\}_{n=1}^{\infty}$ the set of all geodesics on the boundaries of the pairs of
pants and let $\ell (\alpha_n)$ be the length of $\alpha_n$ {with} the hyperbolic metric on $X$. For each $\alpha_n$, define the twist parameter $t(\alpha_n)$ to be the relative distance between the foots of the two  orthogeodesics from $\alpha_n$ to the closest geodesics  in the decomposition on both sides of $\alpha_n$. Since we take relative twists, we can normalize such that $-\frac{1}{2}\leq t(\alpha_n)\leq \frac{1}{2}$, where $\frac{1}{2}$ and $-\frac{1}{2}$ represent the same choice. The Fenchel-Nielsen parameters  $\{ (\ell (\alpha_n),t(\alpha_n)\}_n$ are induced by the hyperbolic metric (\cite{ALPS}). Conversely, a choice of pairs $\{ (\ell_n,t_n)\}_n$ with $\ell_n>0$ and $|t_n|\leq \frac{1}{2}$ defines a surface obtained by gluing with isometry geodesic pairs of pants determined by lengths $\ell_n$ along their boundary curves with choice of $t_n$ uniquely determining the gluings. The obtained metric space may not be complete and in this case we need to add hyperbolic funnels and geodesic half-planes to make a hyperbolic Riemann surface. If the union of the pairs of pants is complete and every geodesic boundary of a pair of pants is glued to another geodesic boundary then the covering group of $X$ is of the first kind (see \cite{BasmajianSaric}).

A Riemann surface $X$ is said to be parabolic if it does not admit a Green's function (\cite{AhlforsSario}). By Ahlfors-Sario \cite{AhlforsSario}, a Riemann surface is {parabolic} if and only if the modulus of the curve family connecting a compact subsurface with the infinity of $X$ is zero. In addition, $X$ is parabolic if and only if the geodesic flow on the unit tangent bundle $T^1(X)$ of $X$ is ergodic if and only if the Poincar\'e series of $\Gamma$ is divergent (see \cite{Nicholls}). 

The type problem for Riemann surfaces is determining when an explicit construction gives rise to a Riemann surface which does not admit a Green's function. This problem has been {studied} extensively using complex analytic constructions of Riemann surfaces. More recently, Basmajian, Hakobyan and the second author \cite{BHS} considered  which conditions on the Fenchel-Nielsen parameters guarantee that the corresponding surface is {parabolic}. A  sufficient condition for an arbitrary Riemann surface {being parabolic} in terms of the lengths of the geodesic boundaries  is given in \cite{BHS}. The twist parameters are not used for arbitrary topological type of Riemann surface. However, the twists become important when the Riemann surface has a small set of ends and large lengths. 

\begin{theorem}[\cite{BHS}]
\label{thm:flutes}
Let $X$ be a tight flute surface with the Fenchel-Nielsen coordinates $\{ (\ell_n,t_n)\}_n$ which correspond to the closed geodesics $\alpha_n$ on the boundary of the pants decomposition.

If 
$$
\sum_{n=1}^{\infty}e^{-\frac{\ell_n}{2}}=\infty
$$
or 
$$
\sum_{n=1}^{\infty}e^{-(1-|t_n|)\frac{\ell_n}{2}}=\infty
$$
then $X$ {is parabolic}.
\end{theorem}

When all twists are zero, the following characterization of {parabolic} flute surfaces holds.

\begin{theorem}[\cite{BHS}]
\label{thm:zero-twist}
Let $X$ be a flute surface with twists $t_n=0$ for all $n$. Then the following are equivalent.
\begin{itemize}
\item $X=\mathbb{H}/\Gamma$ has the covering group $\Gamma$ of the first kind,
\item $\sum_{n=1}^{\infty}e^{-\frac{\ell_n}{2}}=\infty$ and
\item $X$ {is parabolic}.
\end{itemize}

\end{theorem}

\section{Symmetric surfaces and modulus of curve families}

\begin{definition}
An infinite Riemann surface $X$ is called {\it symmetric} if there exists an orientation reversing isometry (anti-conformal reflection) $R:X\to X$ whose set of fixed points consists of pairwise disjoint bi-infinite and/or closed geodesics that divide the surface into two connected components that are permuted by $R$.
\end{definition}

Denote by $R_f\subset X$ the set of fixed points of $R$. If $R$ is an orientation reversing isometry of an infinite Riemann surface $X$ as in the above definition, then $R_f$ has finitely or infinitely many connected components while $X\setminus R_f$ has two components, denoted by $X^*$ and $X^{**}$. Each component is a hyperbolic surface with geodesic boundary such that each boundary component is either a bi-infinite geodesic or a closed geodesic. We call $X^*$ the {\it front side} of $X$.

Consider a family of curves $\Gamma$ in $X$ that are locally rectifiable. Let $\rho$ be a metric on $X$ that is non-negative and Borel measurable. A metric $\rho$ on $X$ is allowable for $\Gamma$ if the $\rho-$length of every curve in $\Gamma$ is at least one. Keep in mind the $\rho-$length of non-rectifiable curves is said to be infinite. We recall the modulus of $\Gamma$ is defined to be
$$
\mathrm{mod}\Gamma_n=\inf_{\rho}\iint_{X}\rho^2(z)dxdy
$$
with the infimum being over all allowable metrics $\rho$ for $\Gamma$.

A symmetry of $X$ suggests a comparison of the modulus of curve families in $X$ to the modulus of curve families in $X^*$. We establish an asymptotic comparison between the modulus of a sequence of curve families connecting a fixed compact subset of $X$ to infinity to the modulus of an analogous sequence in the front side $X^*$. This will be used later to establish a necessary and sufficient condition for the symmetric surfaces to {be parabolic}.

\begin{theorem}
\label{thm:symm-parabolic}
Consider  an infinite symmetric Riemann surface $X$ with an orientation reversing isometry $R:X\to X$. Let $\{X_n\}_{n=1}^\infty$ be an exhaustion of $X$ by finite area subsurfaces with compact geodesic boundary that are invariant under $R$. Denote by $\Gamma_n$ the family of curves that connects $\partial X_1$ and $\partial X_n$ inside $X_n\setminus X_1$.  Let $\Gamma_n^*$ be the subfamily of $\Gamma_n$ that lies in the front side $X^*$ of $X\setminus f(R)$. Then
$$
\lim_{n\to\infty} \mathrm{mod}\Gamma_n=0
$$
if and only if
$$
\lim_{n\to\infty} \mathrm{mod}\Gamma_n^*=0.
$$
\end{theorem}

\begin{proof}
Since $\Gamma_n^*$ is a subfamily of $\Gamma_n$, we have $\mathrm{mod}\Gamma_n^*\leq\mathrm{mod}\Gamma_n$. Thus $\lim_{n\to\infty} \mathrm{mod}\Gamma_n=0$ implies $\lim_{n\to\infty} \mathrm{mod}\Gamma_n^*=0$.

We need to prove the opposite direction. Let $X_n^*=X_n\cap X^*$ be the front side of $X_n$ and assume $\lim_{n\to\infty} \mathrm{mod}\Gamma_n^*=0$. Then there exists   an allowable metric $\rho_n^*$ for $\Gamma_n^*$ on $X_n^*$ such that $$\iint_{X_n^*}[\rho_n^*(z)]^2dxdy\to 0$$ as $n\to\infty$. 

We define $\rho_n(z):=\rho_n^*(z)$ for $z\in X_n^*$, $\rho_n(z):=\rho_n^*(R(z))$ for $z\in X_n\cap X^{**}$ and $\rho_n(z)=\infty$ for $z\in X_n\cap R_f$. 

We claim that $\rho_n(z)$ is allowable for $\Gamma_n$. Indeed, let $\gamma\in\Gamma_n$. If $\gamma\subset X_n^*$ then $\int_{\gamma}\rho_n(z)|dz|=\int_{\gamma}\rho_n^*(z)|dz|\geq 1$. If $\gamma\in\Gamma_n$ intersects $R_f$ in a set of positive length then $\int_{\gamma}\rho_n(z)|dz|=\infty >1$. If $\gamma \in\Gamma_n$ intersects $X_n\cap X^{**}$ and it does not intersect $R_f$ in a set of positive length, then we define $\gamma^*\in\Gamma_n^*$ by mapping with $R$ each component of $\gamma\cap (X_n\cap X^{**})$ to an arc in $X_n^*$ and keeping the other points of $\gamma$ fixed. Then we have $\int_{\gamma^*}\rho_n^*(z)|dz|\geq 1$ because $\rho_n^*$ is allowable for $\Gamma_n^*$ and $\gamma^*\in\Gamma_n^*$. By the definition of $\rho_n$ and $\gamma^*$, we have that $\int_{\gamma^*}\rho_n^*(z)|dz|=\int_{\gamma}\rho_n(z)|dz|$ and $\rho_n$ is allowable for $\Gamma_n$. 

By the definition of $\rho_n$, we have that
$$
\iint_{X_n}\rho_n(z)^2dxdy=2\iint_{X_n^*}\rho_n^*(z)^2dxdy\to 0
$$
as $n\to\infty$. Therefore $\lim_{n\to\infty}\mathrm{mod}\Gamma_n=0$ and the theorem is proved.
\end{proof}

\section{The equivalence of completeness and parabolicity for surfaces with finite ends}

By \cite[Proposition 3.1]{BasmajianSaric}, a Riemann surface $X$ has a Fuchsian covering group of the first kind if and only if every topological pants decomposition of $X$ can be straightened into a geodesic pants decomposition. When constructing a Riemann surface $X$ by gluing infinitely many geodesic pairs of pants, the surface $X$ with the induced metric is complete if and only if the Fuchsian covering group is of the first kind. For this reason we may occasionally call Riemann surfaces with Fuchsian group of the first kind {\it complete} as in the union of pairs of pants make a complete metric space.

We recall that there exist Riemann surfaces whose covering Fuchsian groups are of the first kind that are not parabolic. 
Examples of such surfaces are abundant when the space of topological ends is large, e.g, a Cantor set. McMullen \cite{McMullen} proved that when the Cantor tree surface has a geodesic pants decomposition with cuff lengths bounded below and above by two positive constants then the surface is not parabolic (i.e. it admits a Green's function) while its Fuchsian group is of the first kind. In fact, by \cite[Theorem 1.4]{Saric} the lengths of cuffs can go to zero in a controlled fashion and the same property will hold. 

Therefore, it is of interest to find under which conditions, for surfaces with small space of ends, completeness implies parabolicity. We use Theorem \ref{thm:symm-parabolic} to give a sufficient condition which will extend previously obtained results in \cite[Theorems 1.5 and 1.7]{BHS}.

\subsection{The half-twist infinite flute surfaces} Recall that a flute surface is a planar hyperbolic surface whose space of ends  consists of countably many punctures that accumulate to a single (non-isolated) topological end.
It is known that there exist flute surfaces which have covering Fuchsian groups of the first kind that are not {parabolic}. Conformally, such surfaces are obtained by puncturing the unit disk at countably many points with punctures accumulating to all points of the unit circle (see Kinjo \cite{Kinjo} and \cite[Proposition 4.3]{BasmajianSaric}). When the flute surface is obtained by gluing geodesic pairs of pants with zero twists (see Figure {2}), then the flute surface is complete (equivalently, the covering group is of the first kind) if and only if it is parabolic (see \cite[Theorem 1.5]{BHS}). 

 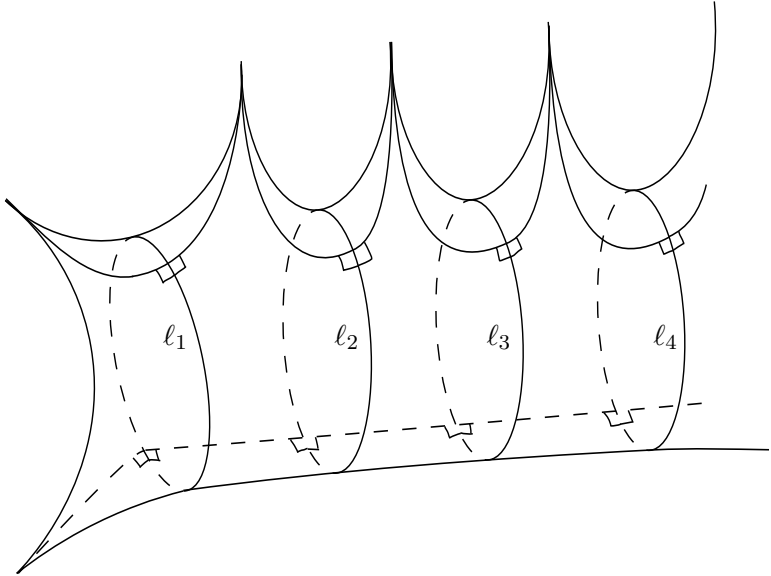
\begin{figure}[h]
\leavevmode \SetLabels
\L(.26*.4) $\ell_1$\\
\L(.44*.4) $\ell_2$\\
\L(.6*.4) $\ell_3$\\
\L(.775*.4) $\ell_4$\\
\endSetLabels
\begin{center}
\AffixLabels{\centerline{\epsfig{file =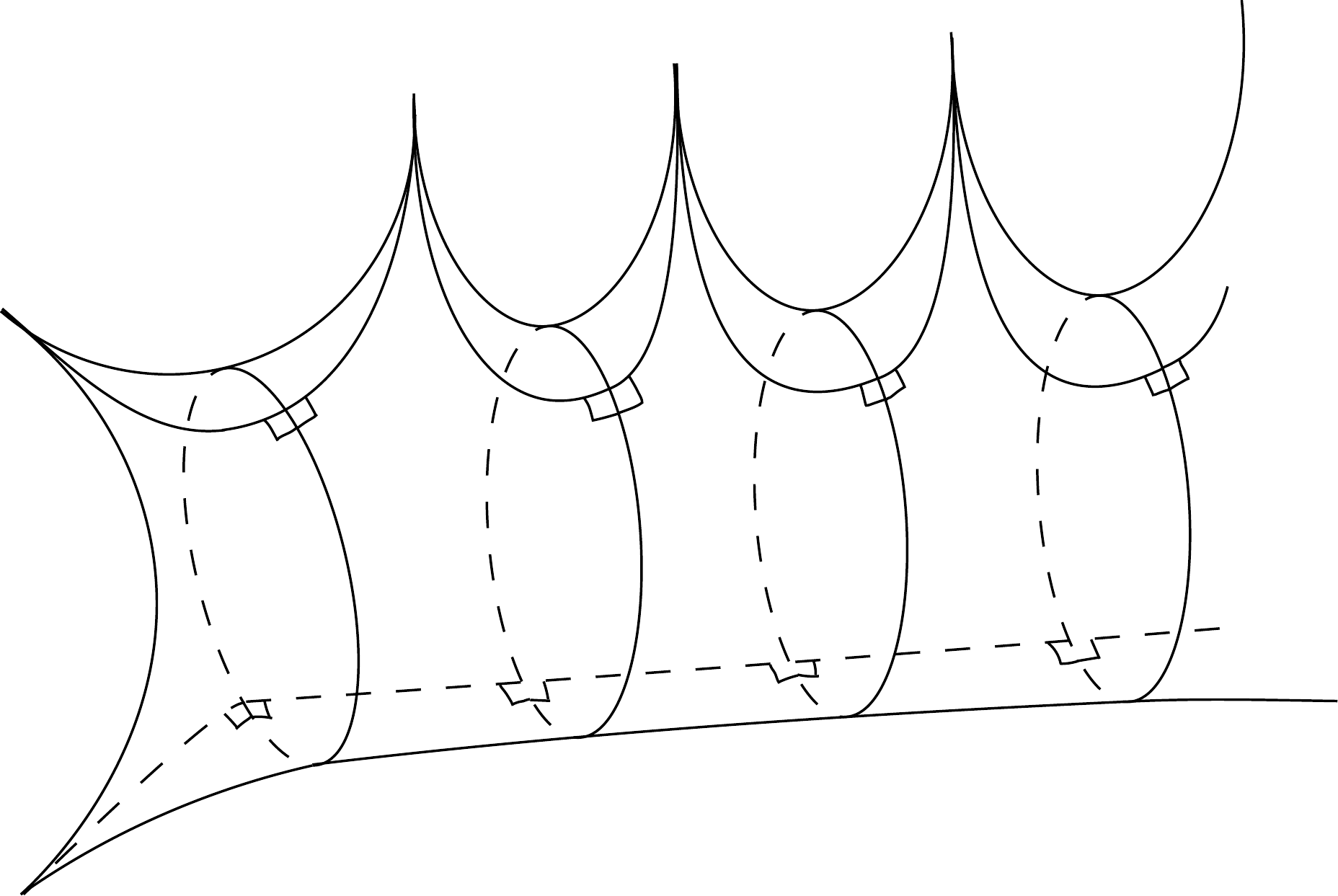,width=4in,height=3in,angle=0} }}
\vspace{-20pt}
\end{center}
\caption{The zero twist flute surface.}
\end{figure}

When the flute surface $X$ is obtained by fixing all twists $t_n\equiv 1/2$, it is called a {\it half-twist flute surface} (see Figure {3}). In \cite[Theorem 1.7]{BHS}, it is proved that if the lengths $\ell_n$ are concave and increasing then $X$ is complete if and only if $X$ {is parabolic}. Using different methods, we extend the validity of this result to all flute surfaces with twists $t_n\in\{ 0,1/2\}$ regardless of the convexity or size of  $\ell_n$. 

 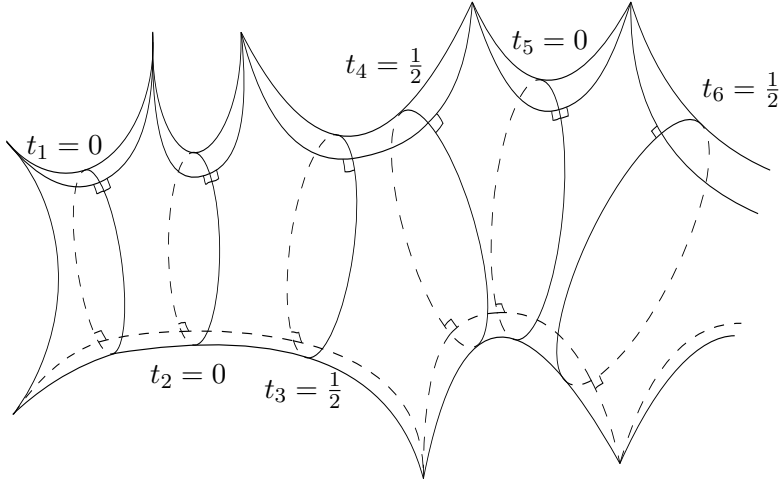
\begin{figure}[h]
\leavevmode \SetLabels
\L(.12*.68) $t_1=0$\\
\L(.25*.2) $t_2=0$\\
\L(.37*.17) $t_3=\frac{1}{2}$\\
\L(.455*.85) $t_4=\frac{1}{2}$\\
\L(.63*.9) $t_5=0$\\
\L(.83*.8) $t_6=\frac{1}{2}$\\
\endSetLabels
\begin{center}
\AffixLabels{\centerline{\epsfig{file =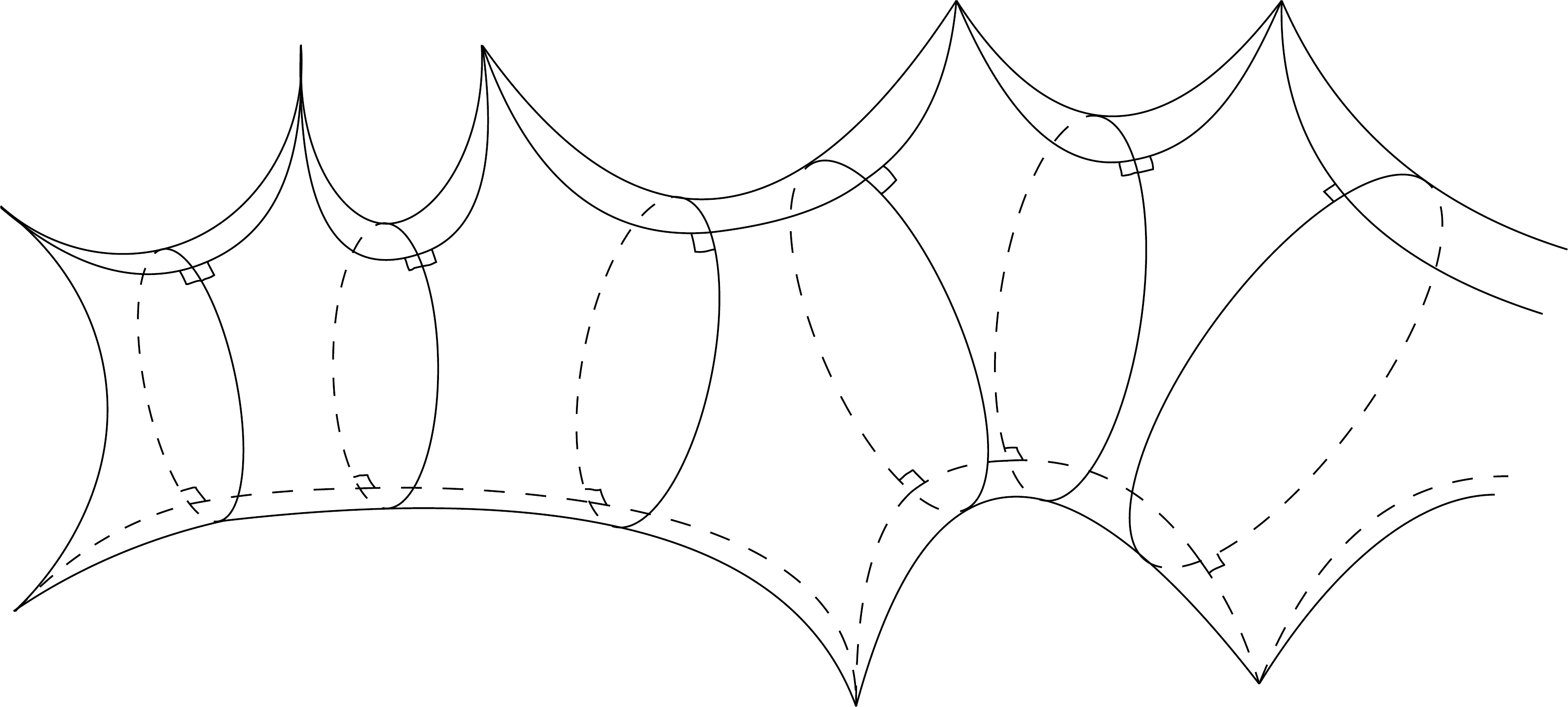,height=2.5in,width=4in,angle=0}}}
\vspace{-30pt}
\end{center}
\caption{The half and zero twists flute surface.} 
\end{figure}

\begin{theorem}
\label{thm:equiv-par-com}
Let $X$ be a flute surface whose twists satisfy $t_n\in\{ 0, 1/2\}$ for all $n$. The following are equivalent:
\begin{itemize}
\item the covering group of $X$ is of the first kind, i.e. $X$ is the union of geodesic pairs of pants without funnels or half-planes,
\item $X$ does not admit {a} Green's function, i.e. $X$ {is parabolic}.
\end{itemize}
\end{theorem}

\begin{proof}
We first partition $X$ into tight geodesic pairs of pants as in Figure 3 and divide each pair of pants into front and back  geodesic pentagons, where pentagons have four right angles and one zero angle.  
Divide $X$ into front side $X^*$ and back side $X^{**}$ by drawing bi-infinite geodesics connecting appropriate cusps as in Figure 3. Since all twists are in $\{ 0,1/2\}$, it follows that the reflection in each pair of pants that maps front to back pentagons extends to a global orientation reversing isometry that preserves the set of bi-infinite geodesics. 

Let $X_n$ be the union of the first $n$ geodesic pairs of pants of $X$ as in Figure 3. Then each $X_n$ is a finite area hyperbolic surface with a single closed geodesic on its boundary of length $\ell_n$. Let $\Gamma_n$ be the curve family in $X_n\setminus X_1$ that connects the boundary components $\partial X_1$ and $\partial X_n$. Ahlfors-Sario \cite[Page 229]{AhlforsSario} and its extension \cite[Proposition 7.3]{BHS} states that $X$ {is parabolic} if and only if $\lim_{n\to\infty}\mathrm{mod}\Gamma_n=0$. Denote by $\Gamma_n^*$  a  subfamily of $\Gamma_n$ that lies in $X^*\cap X_n$. 
By Theorem \ref{thm:symm-parabolic},  $X$ {is parabolic} if and only if $\lim_{n\to\infty}\mathrm{mod}\Gamma_n^*=0$. 

The front side $X^*$ is simply connected and it has a single lift to the universal covering $\mathbb{D}$ that is isometric to it. The lift $\tilde{X}^*$ is an ideal polygon with infinitely many sides and each cusp of $X$ corresponds to a vertex of $\tilde{X}^*$ on $S^1$ (see Figure 4). The edges of the infinite polygon $\tilde{X}^*$ may accumulate to a single point on $S^1$ or to two points of $S^1$. If the accumulation is to two points on $S^1$ then the arc of $S^1$ between the two points is also accumulated by the polygon and the covering group of the surface $X$ is not of the first kind. 

 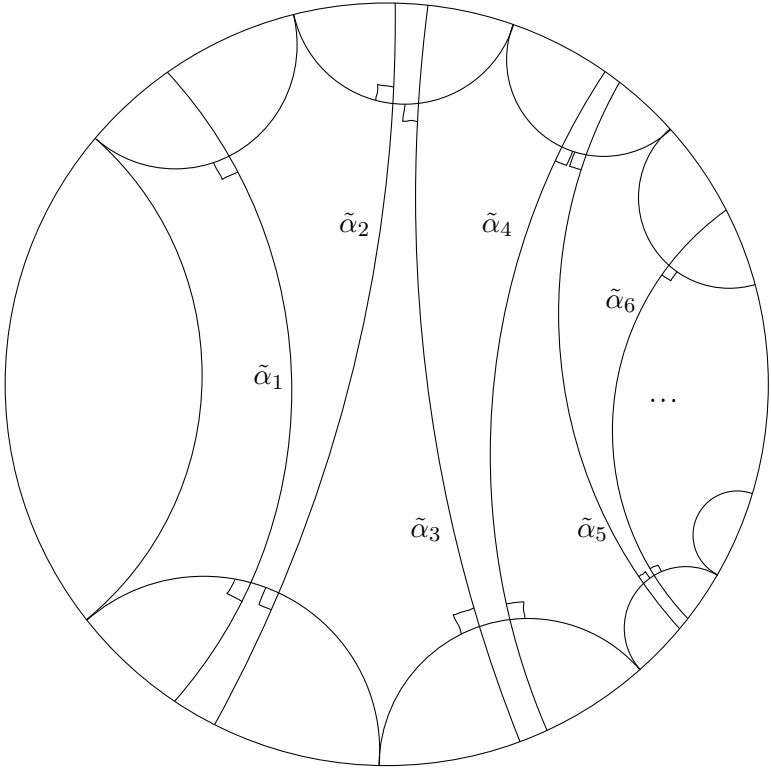
\begin{figure}[h]
\leavevmode \SetLabels
\L(.36*.5) $\tilde{\alpha}_1$\\
\L(.45*.7) $\tilde{\alpha}_2$\\
\L(.525*.3) $\tilde{\alpha}_3$\\
\L(.6*.7) $\tilde{\alpha}_4$\\
\L(.7*.3) $\tilde{\alpha}_5$\\
\L(.73*.6) $\tilde{\alpha}_6$\\
\endSetLabels
\begin{center}
\AffixLabels{\centerline{\epsfig{file =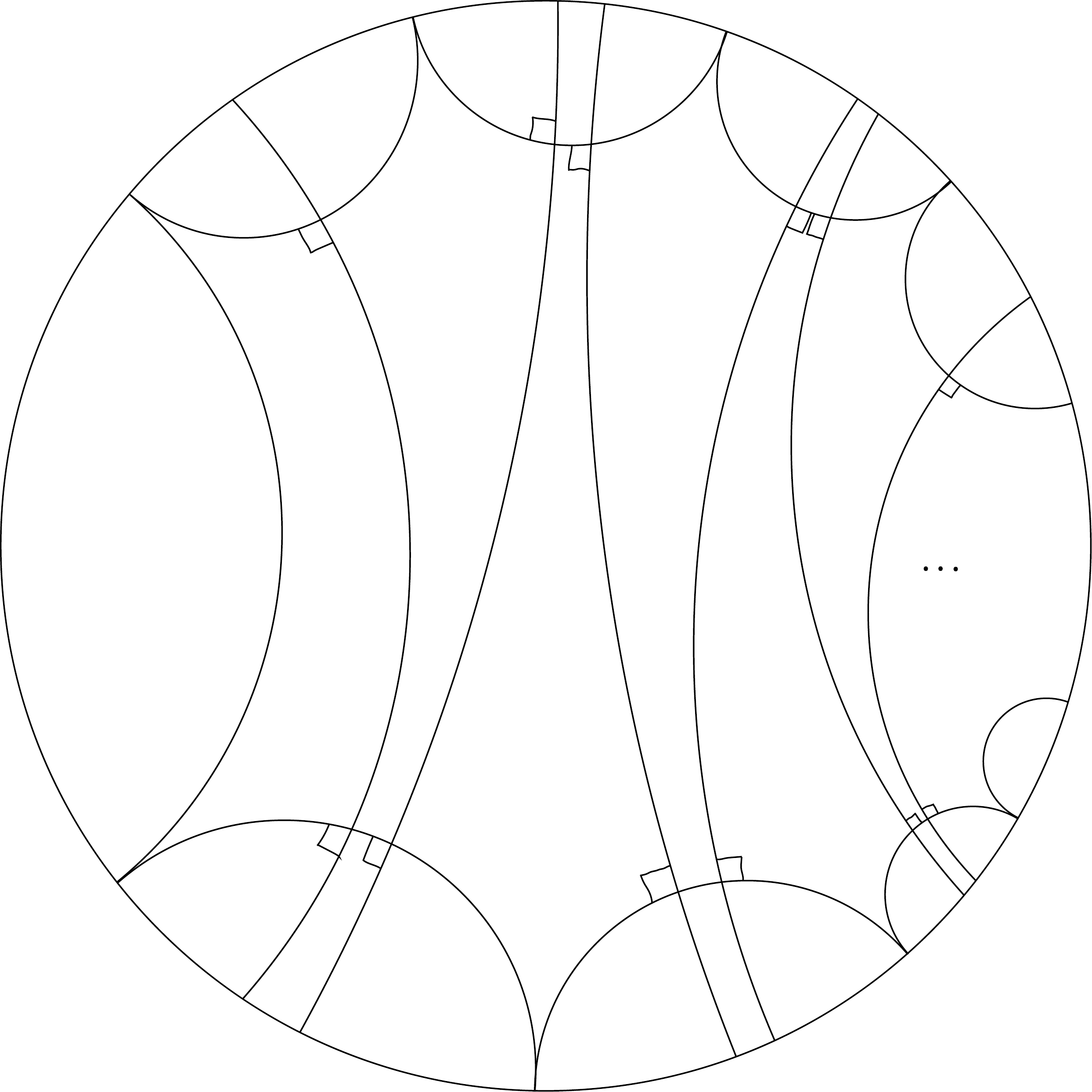,width=4in,width=4in,angle=0}}}
\vspace{-30pt}
\end{center}
\caption{Infinite polygon $\tilde{X}^*$ in $\mathbb{D}$ isometric to $X^*$.} 
\end{figure}

If the accumulation is to one point on $S^1$, we claim that the covering group is of the first kind. To see this, note that the covering group is of the first kind if and only if any geodesic ray starting in $X_1$ and exiting the union of the pairs of pants has infinite length (see \cite{BasmajianSaric}). Assume on the contrary that there exists a geodesic ray $r$ exiting each pair of pants {of} finite length. Then  
we map the connected components of the ray $r$ in $X^{**}$ by symmetry $R$ to geodesic arcs in $X^*$. We obtained a piecewise geodesic arc $r^*$ in $X^*$ which leaves every front pentagon and has finite length. 

Let $\tilde{r}^*$ be the lift of $r^*$ to $\mathbb{D}$ that is contained in the infinite sided polygon that is a lift of $X^*$. The closure of $\tilde{r}^*$ is a compact subset of $\mathbb{D}$. The geodesics $\tilde{\alpha}_n$ which are lifts of the cuffs $\alpha_n$ intersect $\tilde{r}^*$ and therefore converge to a geodesic $\tilde{\alpha}_{\infty}$. The endpoints of $\tilde{\alpha}_{\infty}$ are the accumulations of the sides of $\tilde{X}^*$ which contradicts the assumption that the accumulation is one point on $S^1$. Therefore the covering group is of the first kind.

Thus the covering group is of the first kind if and only if the sides of the infinite polygon accumulate to one point on $S^1$.

If the covering group is not of the first kind then $X$ {is not parabolic}. It remains to be proved that if the polygon accumulates to a single point on $S^1$ then $X$ {is parabolic}. Indeed, the family of curves $\Gamma_n^*$ is converging to the family of curves that have one endpoint equal to the accumulation point of the infinite polygon. It is well-known that the curve family in the plane that goes through a point has zero modulus (see \cite{Garnett}). Thus $X$ is parabolic in this case by Theorem \ref{thm:symm-parabolic}.
\end{proof}

As a  direct corollary to the above proof we obtain

\begin{corollary}
\label{cor:complete-one-sided}
A symmetric flute surface $X$ has covering group of the first kind if and only if, in addition to its ideal vertices, the infinite ideal polygon in $\mathbb{D}$ that is a lift of the front side of $X$ accumulates to a single point on $S^1$.
\end{corollary}

\subsection{Parabolicity of surfaces with finitely many ends}
\label{sec:finite-ends-par}
In this subsection we will consider Riemann surfaces that have infinite genus and finitely many non-planar topological ends. In other words, we assume that $X$ is a Riemann surface obtained by taking a finite area hyperbolic surface $X_0$ whose boundary consists of finitely many closed geodesics $\{\delta_1,\ldots ,\delta_i\}$ and gluing to each boundary component $\delta_i$ a hyperbolic surface $X_i$ with one boundary geodesic and infinite genus that accumulates to a single topological end (see Figure 5). The attached surfaces $X_i$ can be thought of as tails of infinite Loch-Ness monster surfaces.

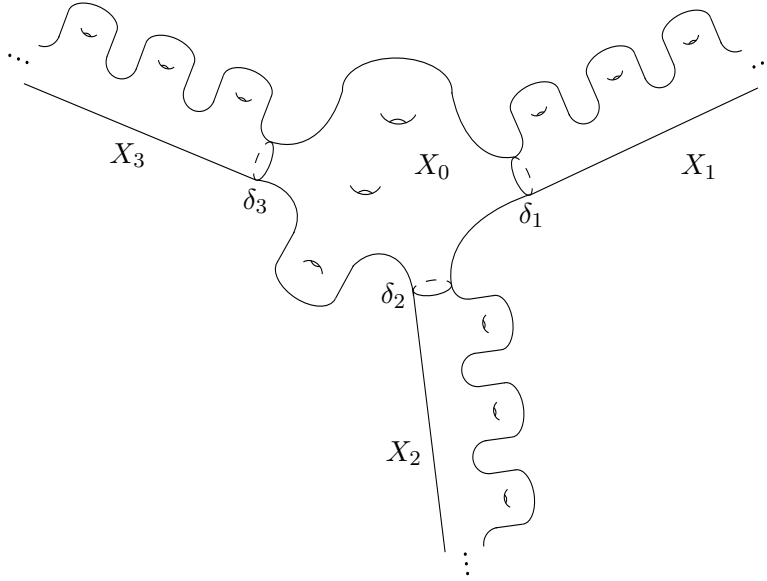
\begin{figure}[h]
\leavevmode \SetLabels
\L(.52*.7) $X_0$\\
\L(.8*.7) $X_1$\\
\L(.49*.2) $X_2$\\
\L(.2*.72) $X_3$\\
\L(.63*.615) $\delta_1$\\
\L(.485*.475) $\delta_2$\\
\L(.34*.64) $\delta_3$\\
\endSetLabels
\begin{center}
\AffixLabels{\centerline{\epsfig{file =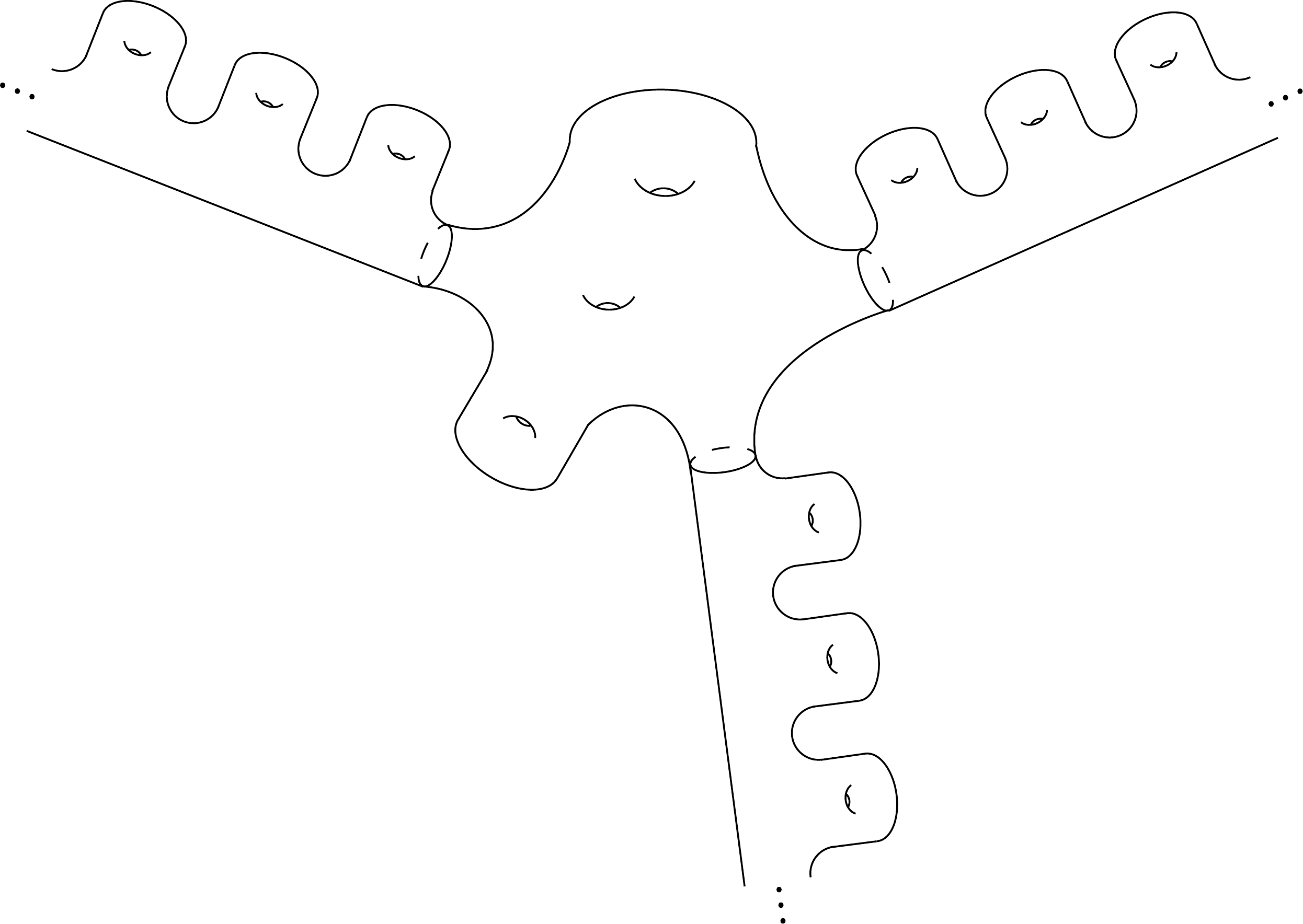,width=4in,height=3in,angle=0} }}
\vspace{-20pt}
\end{center}
\caption{A surface with finitely many non-planar ends.} 
\end{figure}

For each attached surface $X_i$, let $\beta_n$ be closed geodesics that cut off a  genus one surface with geodesic boundary and let $\alpha_n$ be closed geodesics that accumulate towards the infinite end (see Figure 6). We also add another simple closed geodesic $\gamma_n$ in each torus cut off by $\beta_n$. The surface $X_i$ is divided into geodesic pairs of pants by the family of geodesics $\{\delta_i\}\cup \{\alpha_n,\beta_n,\gamma_n\}_n$. Divide each pair of pants of $X_i$ into two right angled hexagons using orthogeodesic arcs between cuffs $\{\delta_k\}\cup \{\alpha_n,\beta_n,\gamma_n\}_n$.

 \begin{figure}[h]
\leavevmode \SetLabels
\L(.05*.2) $\delta_i$\\
\L(.2*.4) $\beta_1$\\
\L(.45*.4) $\beta_2$\\
\L(.69*.4) $\beta_3$\\
\L(.325*.2) $\alpha_1$\\
\L(.56*.2) $\alpha_2$\\
\L(.787*.2) $\alpha_3$\\
\L(.245*.84) $\gamma_1$\\
\L(.475*.84) $\gamma_2$\\
\L(.705*.84) $\gamma_3$\\
\endSetLabels
\begin{center}
\AffixLabels{\centerline{\epsfig{file =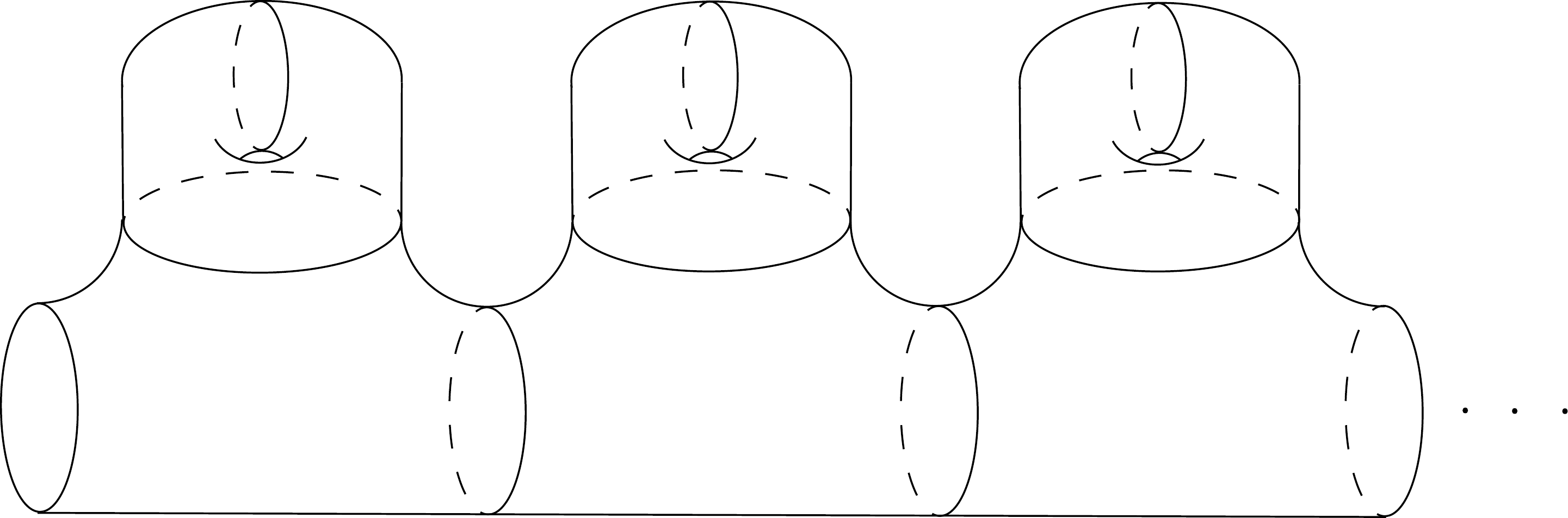,width=4in,height=2in,angle=0} }}
\vspace{-20pt}
\end{center}

\caption{The pants decomposition of an end surface.} 
\end{figure}
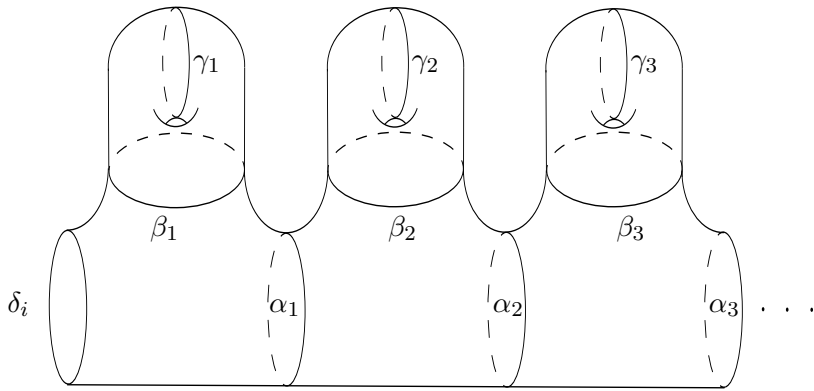

 If the foots of the orthogeodesics from one side of $ \{\alpha_n,\beta_n,\gamma_n\}_n$ meet foots of the {orthogeodesics} from the other side of an end surface $X_i$ then there is a natural decomposition of $X_i$ into front side $X_i^*$ and back side $X_i^{**}$. Choose one of the two hexagons in the pair of pants with cuff $\delta_i$. Then there is a hexagon in the next pair of pants across $\alpha_1$ that shares a side with the front hexagon. This second hexagon is called a front hexagon and we continue in this manner across all $\alpha_n$ to obtain a family of front hexagons converging to the end of $X_i$. For each $\beta_n$ there is a unique hexagon across $\beta_n$ that meets the front hexagons and we will also call {those} hexagons front hexagons. The union of all front hexagons make a connected subsurface $X_i^*$ and the union of the back hexagons makes the complementary connected subsurface $X_i^{**}$. 
 By the construction, there is an orientation reversing isometry $R$ of $X_i$ such that $R(X_i^*)=X_i^{**}$ which maps front hexagons onto back hexagons (see Figure 7). In this situation, the end surface is said to be {\it symmetric}. The surface $X$ is said to have {\it symmetric ends} if every $X_i$ is symmetric.

 \begin{figure}[h]
\begin{center}
\AffixLabels{\centerline{\epsfig{file =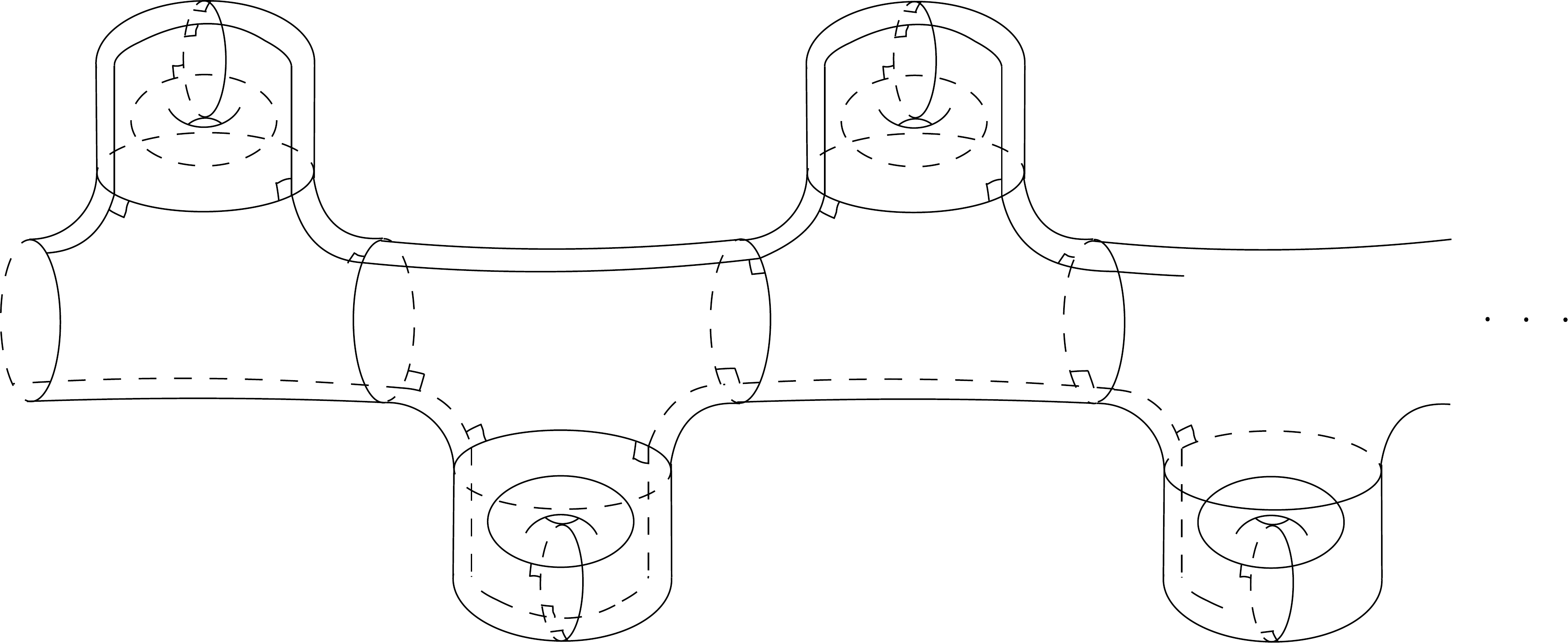,width=4.5in,height=2in,angle=0} }}
\caption{A symmetric end surface with half twists.} 
\end{center}
\end{figure}

\begin{theorem}
\label{thm:symmetric-finite-ends}
Let $X$ be a Riemann surface with finitely many symmetric non-planar ends such that the lengths of the geodesics $\beta_n$ and $\gamma_n$ from Figure 6 are between two positive constants. Then the following are equivalent:
\begin{itemize}
\item the Fuchsian covering group of $X$ is of the first kind,
\item the Riemann surface $X$ admits no Green's function. 
\end{itemize}
\end{theorem}

\begin{remark}
Note that {the} $\alpha_n$ are not assumed to be bounded. In fact, the most interesting applications are when the lengths of $\alpha_n$ go to infinity.
\end{remark}

\begin{proof}
It is enough to prove that if the covering group of $X$ is of the first kind then $X$ {is parabolic}. 
Under the assumptions on $X$, each end surface $X_i$ is symmetric and the two sides $X_i^*$ and $X_i^{**}$ are planar but not simply connected.

We define a compact exhaustion of $X$ by taking $X^n$ to be the compact subsurface  of $X$ whose boundary components are geodesics $\alpha_n$ in each $X_i$ for $i=1,\ldots ,k$.  
Let $\Gamma^n$ be the family of curves in $X^n\setminus X_0$ that connects the boundary of $ X_0$ with the boundary of $X^n$. By Ahlfors-Sario \cite{AhlforsSario}, it is enough to prove that $\lim_{n\to\infty}\mathrm{mod}\Gamma^n =0$. Denote by $\Gamma_i^n$ the subfamily of $\Gamma^n$ that lies in $X_i^n$. Then $\Gamma^n$ is a disjoint union of $\Gamma_i^n$ for $i=1,\ldots ,k$. Since the supports of $\Gamma_i^n$ are pairwise disjoint, we have
$$
\mathrm{mod}\Gamma^n =\sum_{i=1}^k\mathrm{mod}\Gamma_i^n.
$$
Therefore, it is enough to prove that $\lim_{n\to\infty}\mathrm{mod}\Gamma_i^n=0$ for all $i=1,\ldots ,k$. 

We recall that the modulus is a quasiconformal quasi-invariant. Therefore $\mathrm{mod}\Gamma_i^n$ tends to zero on $X_i$ if and only if it is tends to zero on a quasiconformal image of $X_i$. Since the lengths of $\beta_n$ and $\gamma_n$ are between two positive constants, there exists a quasiconformal map onto an infinite surface where $\beta_n$, $\gamma_n$ and the twists on these closed geodesics are all equal to zero (see Shiga \cite{Shiga}). Without loss of generality, we can  assume that all tori cut off by the $\beta_n$ in $X_i$ are isometric.

The {front of the surface} $X_i^*$  is not simply connected (see Figure 8). However,  $X_i^*\setminus (\cup_n\gamma_n)$ is simply connected and it has a lift to the universal cover as in Figure 8. Let $X_i^n=X^n\cap X_i$ and $X_i^{n*}=X_i^n\cap X_i^*$ be the front side of $X_i^n$. Let $\Gamma_i^{n*}$ be the subfamily of $\Gamma_i^n$ that lies entirely in $X_i^{n*}$. By Theorem \ref{thm:symm-parabolic} it is enough to prove that, for all $i=1,\ldots ,k$,
$$
\lim_{n\to\infty}\mathrm{mod}\Gamma_i^{n*}=0.
$$

 \begin{figure}[h]
\begin{center}
\AffixLabels{\centerline{\epsfig{file =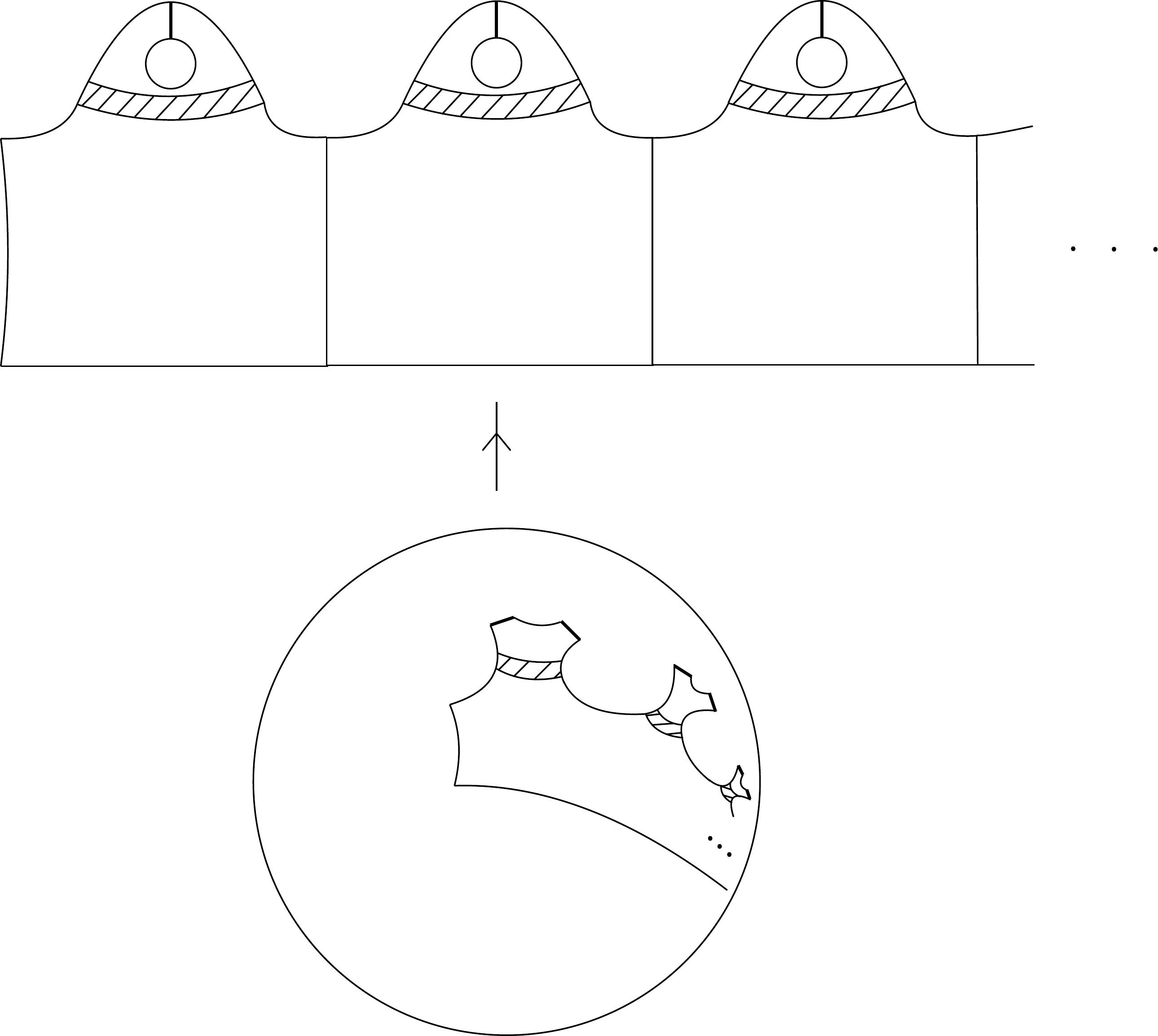,width=4in,height=3in,angle=0} }}
\end{center}
\caption{The front side $X_i^*$ and its lift to the universal covering.} 
\end{figure}

Since $X$ has covering group of the first kind, it follows that a single component of the lift to the universal covering $\mathbb{D}$ of $X_i^*\setminus (\cup_n\gamma_n)$ accumulates to exactly one point on $S^1$. Therefore the family of all curves connecting  $\delta_i\cap X_i^*$ to the end of $X_i^*$ and not intersecting the family $\{ \gamma_n\}_n$ has zero modulus. Indeed, since this curve family lies in $X_i^*\setminus (\cup_n\gamma_n)$ and the lifting map is conformal, it follows that this curve family is mapped by a conformal map to a curve family in the complex plane that passes through a single point. It is classical that the modulus is zero for such curve families (see \cite{Garnett}) and that the limit of the modulus of a sequence of curve families whose curves converge to curves passing through a point is zero.

 However, the above argument does not apply to curves in $X_i^*$ that intersect the family $\{\gamma_n\}_n$. To show that the limit of the modulus of a sequence  is zero for the sequence of curve families  connecting $\delta_i\cap X_i^*$ and $\alpha_n\cap X_i^{n*}$ inside $X_i^{n*}$, we will find a quasiconformal map that maps $X_i^{*}$ into the lift of $X_i^*\setminus (\cup_n\gamma_n)$.  This will finish the proof since a quasiconformal map quasi-preserves the modulus and we already established that the limit of the modulus of the sequence of curve families in $X_i^*\setminus (\cup_n\gamma_n)$ is zero.

To find the quasiconformal map, consider the standard half-collar around each $\beta_n$ that is inside the torus cut off by $\beta_n$. The half-collar is disjoint from $\gamma_n$ and it is  contained in $X_i^*\setminus (\cup_n\gamma_n)$. Let $X_i'$ be $X_i^*$ minus the tori cut off by $\beta_n$ union the half-collars around $\beta_n$. We will find a quasiconformal map $f$ from $X_i^*$ to (but not onto) $X_i'$. We define $f$ to be the identity on $X_i^*$ minus the tori and on $\beta_n\cap X_i^*$. It remains to find a quasiconformal map which maps the front half of the tori cut off by $\beta_n$ to the front half of the half collar that is the identity on $\beta_n\cap X_i^*$. Since all tori are quasiconformal, it is enough to find one such quasiconformal map and use it for all $n$.

 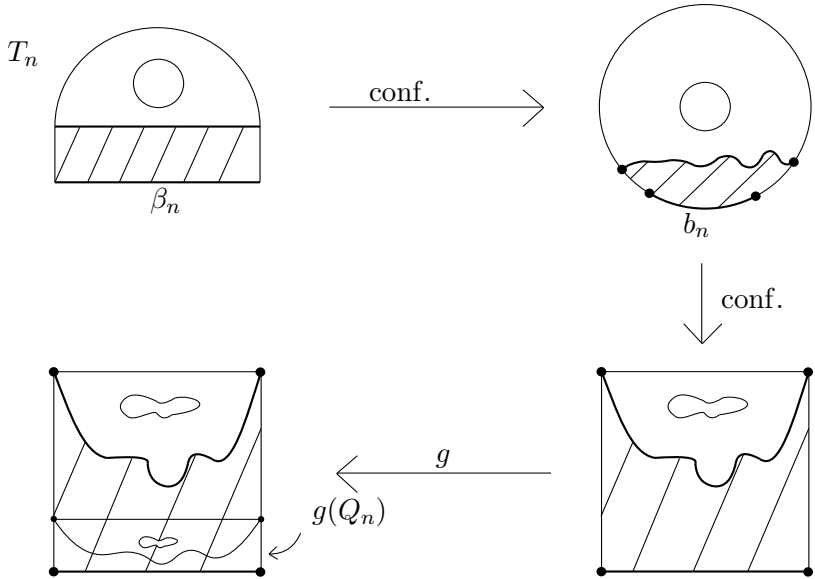
\begin{figure}[h]
\leavevmode \SetLabels
\L(.05*.9) $T_n$\\
\L(.2*.643) $\beta_n$\\
\L(.43*.83) $\mathrm{conf.}$\\
\L(.76*.6) $b_n$\\
\L(.8*.47) $\mathrm{conf.}$\\
\L(.5*.2) $g$\\
\L(.37*.1) $g(Q_n)$\\
\endSetLabels
\begin{center}
\AffixLabels{\centerline{\epsfig{file =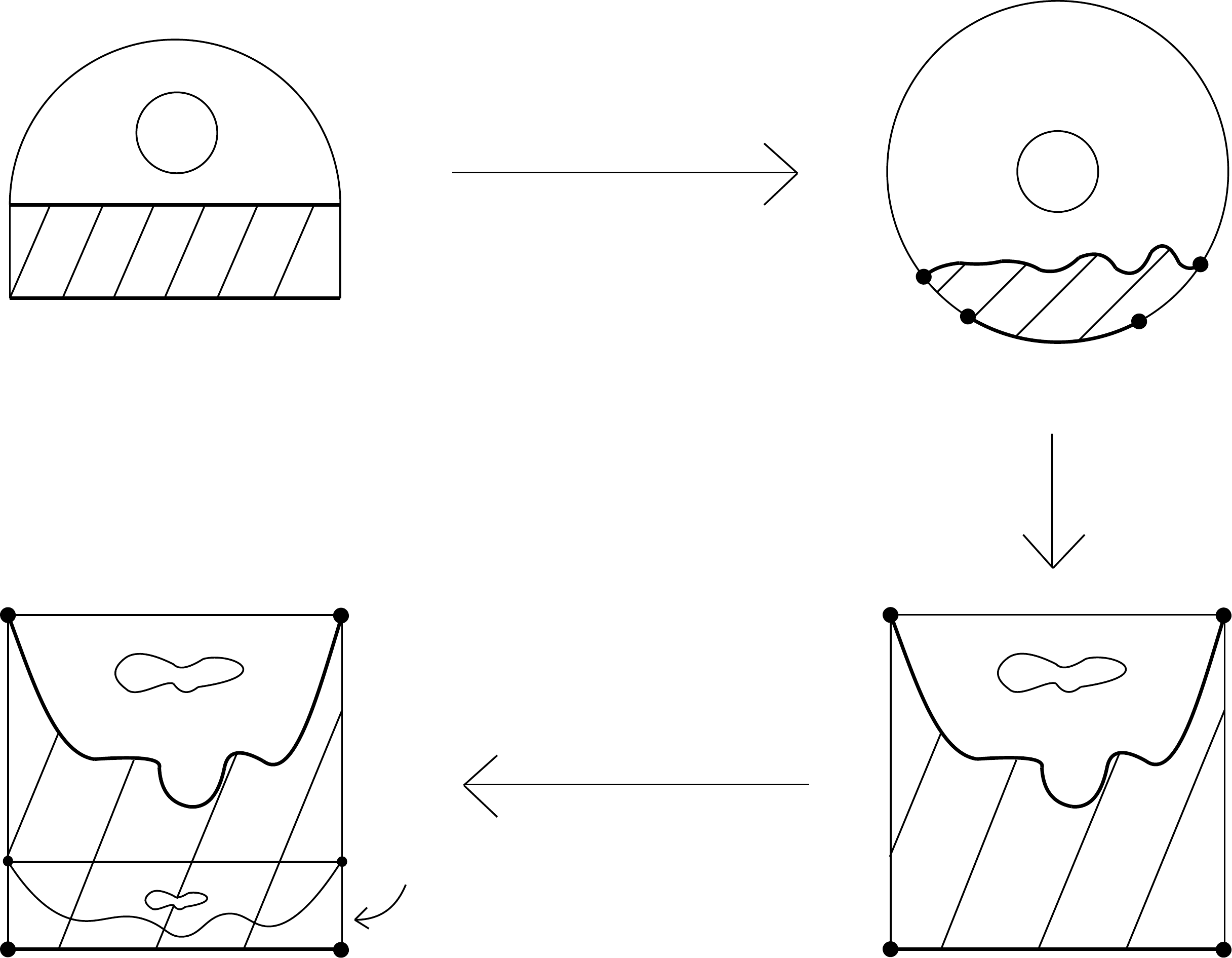,width=4in,height=3in,angle=0} }}
\end{center}
\caption{The quasiconformal map.} 
\end{figure}

The front half $T_n$ of the tori cut off by $\beta_n$ is a doubly connected region whose outside boundary consists of $\beta_n\cap X_i^*$ and an orthogeodesic from $\beta_n$ to itself and the inner boundary consists of the orthogeodesic from $\gamma_n$ to itself (see Figure 9). We map $T_n$ conformally to a Euclidean annulus $\{ z: r\leq |z|\leq 1\}$. Denote the image of $\beta_n\cap X_i^*$ by $b_n\subset \{ |z|=1\}$. The boundary of the half collar is mapped to an arc $c_n$ in $\{ z: r< |z|\le 1\}$ with endpoints in $\{ |z|=1\}\setminus b_n$. The arc $c_n$ separates $b_n$ from the inner boundary $\{ |z|=r\}$. Choose a conformal map from {$\{ z: |z|\le 1\}$} onto a rectangle of height $1$  such that $b_n$ is the bottom horizontal side on the real axis and the endpoints of $c_n$ are the two top vertices. By the abuse of notation, we will denote by $c_n$ the image of $c_n$ in the rectangle. The curve $c_n$ has {one} lowest point with positive height $h$. The map $g$ which shrinks vertically the rectangle by the factor $h$ is mapping the image of $T_n$ in the rectangle under the arc $c_n$ (which is the image of the half-collar). Note that $g$ has quasiconformal constant $1/h$. By conjugating $g$ with the composition of the above conformal maps we obtain a quasiconformal map $f$ that sends $T_n$ to the front of the half-collar around $\beta_n$. This finishes the proof.
\end{proof}

As a  direct corollary to the above proof we obtain

\begin{corollary}
\label{cor:complete-one-sided-finite-ends}
A symmetric  surface $X$ with finitely many end surfaces $X_i$, for $i=1,2,\ldots n$, has covering group of the first kind if and only if, for each $i\in \{1,2,\ldots ,n\}$, the infinite  polygon in $\mathbb{D}$ that is a lift of $X_i^*\setminus (\cup_n\gamma_n)\subset X$ accumulates to a single point on $S^1$.
\end{corollary}

\begin{remark}
\label{rem:nested-alpha-acc}
For a fixed lift of $X_i^*\setminus (\cup_n\gamma_n)$ in $\mathbb{D}$, let $\widetilde{\alpha_n}$ be the lift of the closed geodesic $\alpha_n$ that connects {its} two  boundary sides. Then the accumulation to one point of $S^1$ in the above corollary is equivalent to the accumulation of the nested sequence $\widetilde{\alpha_n}$ to one point of $S^1$.
\end{remark}

\section{The parabolicity of symmetric surfaces via cuff lengths}

In this section we decide when a symmetric Riemann surface is parabolic from its Fenchel-Nielsen coordinates. The result will complement some of the results obtained in \cite{BHS}.

\subsection{The half-twist flute surfaces} Let $X$ be a flute surface with the Fenchel-Nielsen coordinates $\{ (\ell_n,t_n\equiv 1/2)\}_n$. This surface is called a {\it half-twist} flute surface since we always twist by one half. We establish the following theorem:

\begin{theorem}
\label{thm:half-twist-O_G}
Let $X=\{ (\ell_n,1/2)\}_n$ be a half-twist flute surface with increasing sequence of cuff lengths $\ell_n$. Then $X$ is parabolic if and only if
$$
\sum_{n=1}^\infty e^{-\sigma_n/2}=\infty , 
$$
where $\sigma_n=\ell_n-\ell_{n-1}+\cdots +(-1)^{n-1}\ell_1$.
\end{theorem}

Before proving the above theorem, we will establish several lemmas. 
Consider two ideal geodesic triangles $\Delta_1$ and $\Delta_2$ with disjoint interiors and a common boundary side $g$. We orient $g$ such that $\Delta_1$ is to its left and consider the two orthogeodesics to $g$ from the vertices of $\Delta_1$ and $\Delta_2$ not on $g$. 
The {\it shear} $s(g)$ of the configuration $(\Delta_1,\Delta_2)$ is the signed distance (with respect to the orientation of $g$) between the foot on $g$ of the orthogeodesic in $\Delta_1$ to the foot of the orthogeodesic in $\Delta_2$. The shear of the configuration $(\Delta_2,\Delta_1)$ equals the shear of the configuration $(\Delta_1,\Delta_2)$.

Recall that the front side $X^*$ of the half-twist {flute} surface $X$ is simply connected and its single lift to the universal covering $\mathbb{D}$ is an infinite polygon (see Figure 10). We denote by $\tilde{X}^*$ a single lift of $X^*$ to $\mathbb{D}$.
By Corollary \ref{cor:complete-one-sided}, the covering group of $X$ is of the first kind if and only if in addition to the countable sets of vertices, the infinite polygon $\tilde{X}^*$ has only one more point of accumulation on $S^1$.  
 
Therefore we need to establish that $\tilde{X}^*$ has only one extra accumulation point on $S^1$ in addition to its ideal vertices. The fronts of the closed geodesic boundaries (cuffs) $\{\alpha_n\}_n$ of the pants decomposition of $X$ lift to orthogeodesic arcs between the sides of the infinite polygon $\tilde{X}^*$. Denote by $g_{2n-1}$ the geodesic in $\mathbb{D}$ which contains the lift of $\alpha_n$ as in Figure 10. We orient $g_{2n-1}$ such that $g_{2n+1}$ is to its right. Let $g_{2n}$ be the geodesic in $\mathbb{D}$ whose endpoints are the initial point of $g_{2n-1}$ and the terminal endpoint of $g_{2n+1}$ (see Figure 10). We are interested in computing the shears $s(g_n)$ for $n\geq 2$, where $g_n$ is a diagonal of the ideal quadrilateral whose vertices  are the endpoints of $g_{n-1}$ and $g_{n+1}$.
 
  \begin{figure}[h]
\leavevmode \SetLabels
\L(.36*.4) $g_1$\\
\L(.47*.6) $g_2$\\
\L(.57*.9) $g_3$\\
\L(.5*.4) $g_4$\\
\L(.61*.5) $g_5$\\
\L(.7*.6) $g_6$\\
\L(.7*.4) $g_7$\\
\endSetLabels
\begin{center}
\AffixLabels{\centerline{\epsfig{file =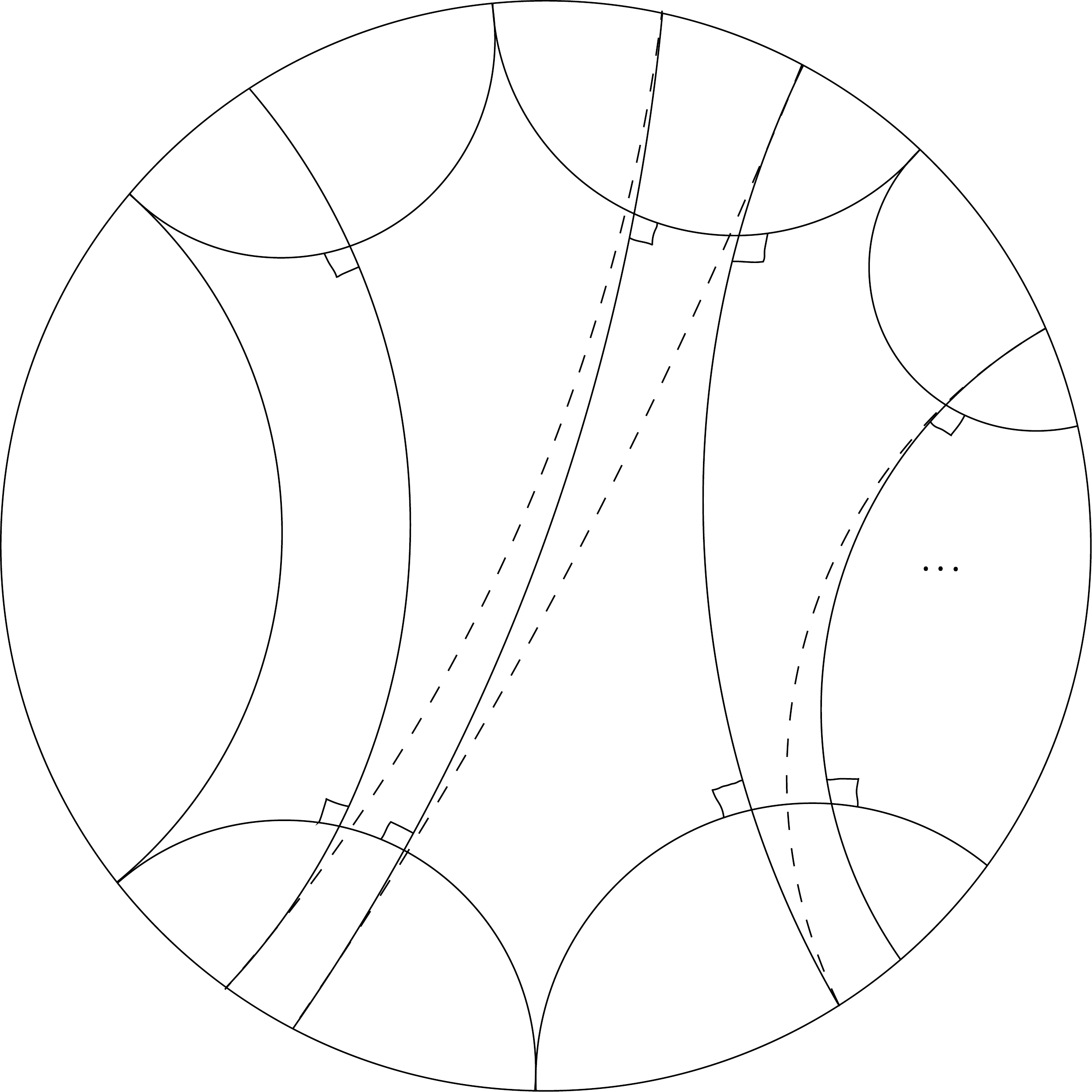,width=3.5in,height=3.5in,angle=0}}}
\vspace{-30pt}
\end{center}
\caption{The infinite polygon $\tilde{X}^*$ and its limit points on $S^1$.} 
\end{figure}
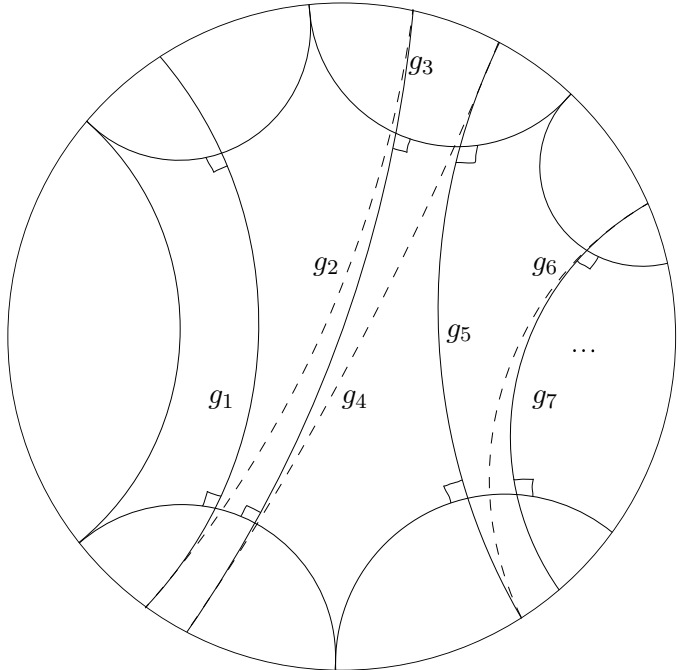

 The common orthogonal geodesic arc $\eta_n$ between $g_{2n-1}$ and $g_{2n+1}$ lies on $\tilde{X}^*$. The common orthogonal $\eta_n$ together with the parts of $g_{2n-1}$ and $g_{2n+1}$ in $\tilde{X}^*$ make three sides of a geodesic pentagon in $\tilde{X}^*$ that has four right angles and a zero angle at the vertex of $\tilde{X}^*$ between $g_{2n-1}$ and $g_{2n+1}$. The other two sides of the pentagon are geodesic rays on the boundary of $\tilde{X}^*$. The lengths of the sides of the pentagon on $g_{2n-1}$ and $g_{2n+1}$ are $\ell_n/2$ and $\ell_{n+1}/2$, respectively. By the formulas for the Saccheri rectangle, we obtain
$$
\ell (\eta_n)=\tanh^{-1}\Big{(}\frac{1}{\cosh \frac{\ell_n}{2}}\Big{)}+\tanh^{-1}\Big{(}\frac{1}{\cosh \frac{\ell_{n+1}}{2}}\Big{)},
$$
which implies, for large n (for the upper estimate) and for a universal constant $C>0$ ($C=8$ works),

\begin{equation}
\label{eq:eta_n}
e^{-\frac{\ell_{n+1}}{2}} < \ell (\eta_n) < Ce^{-\frac{\ell_n}{2}}.
\end{equation}

\begin{lemma}
\label{lem:shear-even}
Under the above notation, the shear along $g_{2n}$ for the quadrilateral whose vertices are at the endpoints of $g_{2n-1}$ and $g_{2n+1}$ is given by
$$
s(g_{2n})=2\log\sinh\frac{\ell (\eta_n)}{2}.
$$
\end{lemma}

\begin{remark}
Note that $s(g_{2n})<0$ for $n$ large enough because $\ell (\eta_n)\to 0$ as $n\to\infty$. 
\end{remark}

\begin{proof}
The distance between $g_{2n-1}$ and $g_{2n+1}$ is $\ell(\eta_n )$. Let $A:\mathbb{D}\to\mathbb{H}$ be a M\"obius map such that $\eta_n$ is mapped onto the $y$-axis.  The geodesics $A(g_{2n-1})$ and $A(g_{2n+1})$ have endpoints $(-x_n,x_n)$ and $(-y_n,y_n)$, respectively. Without loss of generality, we can assume that $x_n<y_n$ and we have
$$
\ell (\eta_n)=\log\frac{y_n}{x_n}.
$$

Let $B:\mathbb{H}\to\mathbb{H}$ be a M\"obius map such that $B(x_n)=0$, $B(-x_n)=-1$ and $B(-y_n)=\infty$. Then $B(z)=\frac{y_n-x_n}{2x_n}\frac{z-x_n}{z+y_n}$,
$B(y_n)>0$ and
$$
s(g_{2n})=\log (B(y_n))=\log \Big{[}\frac{y_n-x_n}{2\sqrt{x_n}\sqrt{y_n}}\Big{]}^2=\log\sinh^2\frac{\ell (\eta_n)}{2}.
$$
\end{proof}

The remaining cases are finding shears on $g_n$ with $n$ odd. The expression is more complicated as it depends on the lengths of two adjacent $\eta_i$ and on the length of the cuff that lifts to $g_n$. The formula will slightly differ for the indices that have remainders $1$ and $3$ under division by $4$. 

\begin{lemma}
\label{lem:shear-odd}
Consider the lift of $X^*$ and geodesics $g_n$ as above. The shear of $g_n$ is defined with respect to the quadrilateral whose vertices are the  ideal endpoints of $g_{n-1}$ and $g_{n+1}$. When $\ell_n$ is large enough, we have
$$
s(g_{4n+1})=\sinh^{-1}\frac{1}{\sinh\ell (\eta_{2n}) }+\sinh^{-1}\frac{1}{\sinh\ell (\eta_{2n+1}) }-\frac{\ell_{2n+1}}{2}
$$
and
$$
s(g_{4n+3})=\sinh^{-1}\frac{1}{\sinh\ell (\eta_{2n+1}) }+\sinh^{-1}\frac{1}{\sinh\ell (\eta_{2n+2}) }+\frac{\ell_{2n+2}}{2}.
$$
\end{lemma}

\begin{proof}
Consider $g_{4n+1}$ and the corresponding quadrilateral as in Figure 11. The geodesic $g_{4n+1}$ is the diagonal and denote by $A$ and $D$ the other two vertices of the quadrilateral on the left and right side of $g_{4n+1}$, correspondingly. Then  $A$ is the initial endpoint of $g_{4n-1}$ and $D$ is the terminal endpoint of $g_{4n+3}$. Let $P\in g_{4n+1}$ be the foot of the orthogeodesic from $A$ to $g_{4n+1}$, and let $S$ be the foot of the orthogeodesic from $D$. Let $B\in g_{4n-1}$ and $Q\in g_{4n+1}$ be the endpoints of the orthogeodesic $\eta_{2n}$ between $g_{4n-1}$ and $g_{4n+1}$, and let $R\in g_{4n+1}$ and $C\in g_{4n+3}$ be the endpoints of the orthogeodesic $\eta_{2n+1}$ between $g_{4n+1}$ and $g_{4n+3}$ (see Figure 11).

 \begin{figure}[h]
\leavevmode \SetLabels
\L(.29*.71) $B$\\
\L(.51*.71) $Q$\\
\L(.465*.62) $S$\\
\L(.51*.43) $P$\\
\L(.47*.34) $R$\\
\L(.61*.31) $C$\\
\L(.18*.2) $A$\\
\L(.82*.72) $D$\\
\L(.51*.1) $g_{4n+1}$\\
\endSetLabels
\begin{center}
\AffixLabels{\centerline{\epsfig{file =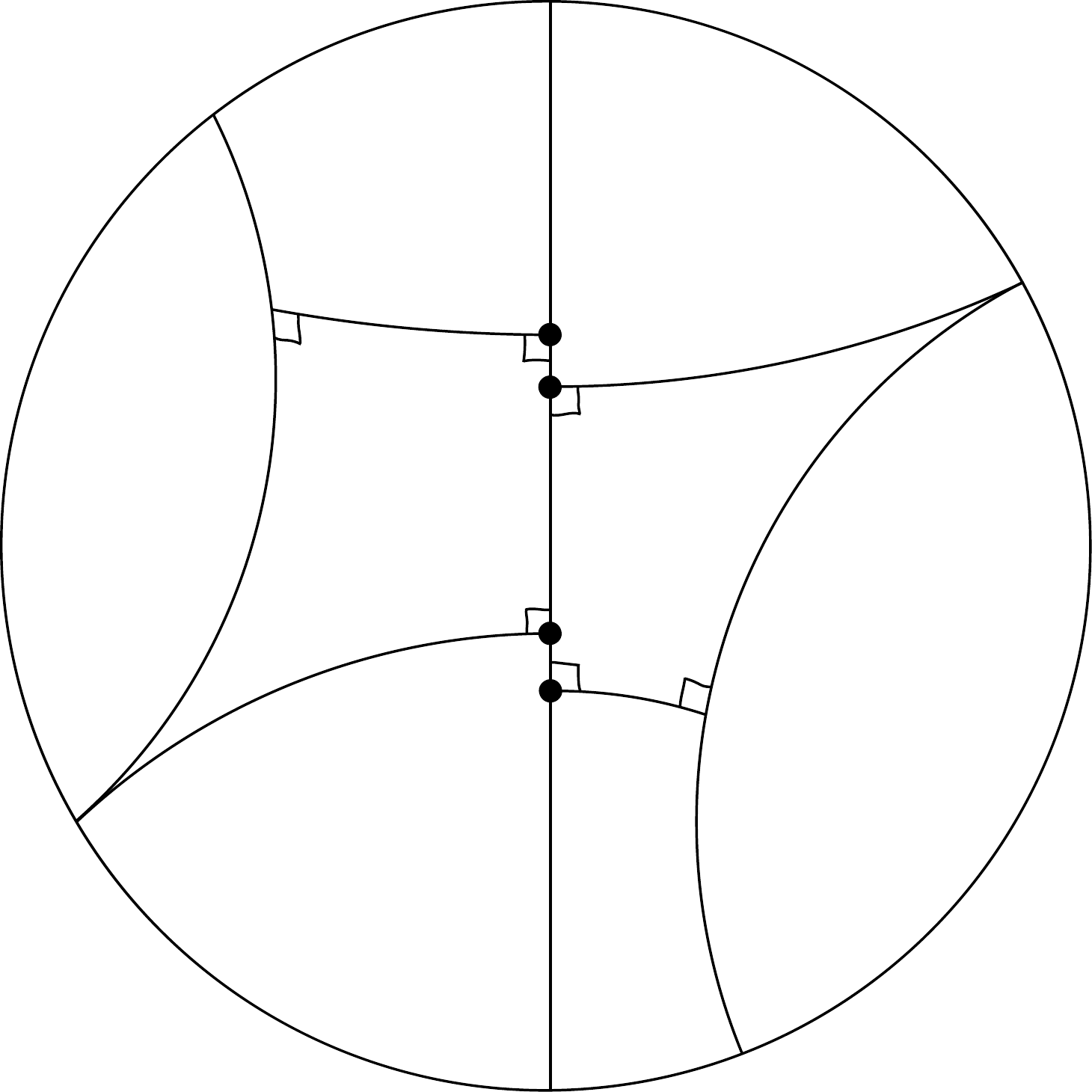,width=3.5in,height=3.5in,angle=0}}}
\vspace{-30pt}
\end{center}
\caption{$s(g_{4n+1})=\ell (PQ)+\ell (RS)-\frac{\ell_{2n+1}}{2}$.} 
\end{figure}
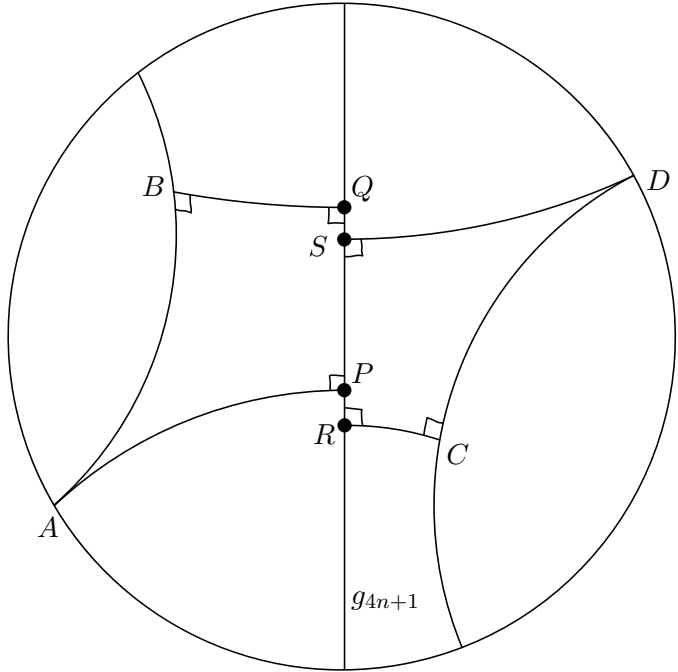

The starting choice of the half-twists and the fact that the index of $g_i$ has remainder $1$ under division by $4$ guarantees that the arc $RQ$ contains arc $PS$ (see Figures 10 and 11). From Figure 11 we obtain 
$$s(g_{4n+1})=\ell (PQ)+\ell (RS)-\frac{\ell_{2n+1}}{2},$$
where  $\ell (PQ)$ and $\ell (RS)$ are the lengths of arcs $PQ$ and $RS$, respectively. From the Lambert quadrilaterals $ABQP$ and $RSDC$, we obtain $\ell (PQ)=\sinh^{-1}\frac{1}{\sinh \ell (\eta_{2n})}$ and $\ell (RS)=\sinh^{-1}\frac{1}{\sinh \ell (\eta_{2n+1})}$ which gives the formula for $s(g_{4n+1})$.

Consider $g_{4n+3}$ and the corresponding quadrilateral as in Figure 12. The geodesic $g_{4n+3}$ is the diagonal and denote by $A$ and $D$ the other two vertices of the quadrilateral on the left and  right side of $g_{4n+3}$, correspondingly. Then  $A$ is the initial endpoint of $g_{4n+1}$ and $D$ is the terminal endpoint of $g_{4(n+1)+1}$. Let $P\in g_{4n+3}$ be the foot of the orthogeodesic from $A$ to $g_{4n+3}$, and let $S$ be the foot of the orthogeodesic from $D$. Let $B\in g_{4n+1}$ and $Q\in g_{4n+3}$ be the endpoints of the orthogeodesic $\eta_{2n+1}$ between $g_{4n+1}$ and $g_{4n+3}$, and let $R\in g_{4n+3}$ and $C\in g_{4(n+1)+1}$ be the endpoints of the orthogeodesic $\eta_{2n+2}$ between $g_{4n+3}$ and $g_{4(n+1)+1}$ (see Figure 12).

 \begin{figure}[h]
\leavevmode \SetLabels
\L(.47*.79) $S$\\
\L(.47*.61) $R$\\
\L(.61*.61) $C$\\
\L(.7*.94) $D$\\
\L(.51*.34) $Q$\\
\L(.475*.15) $P$\\
\L(.36*.3) $B$\\
\L(.32*.01) $A$\\
\L(.3*.5) $g_{4n+1}$\\
\L(.51*.1) $g_{4n+3}$\\
\L(.6*.41) $g_{4(n+1)+1}$\\
\endSetLabels
\begin{center}
\AffixLabels{\centerline{\epsfig{file =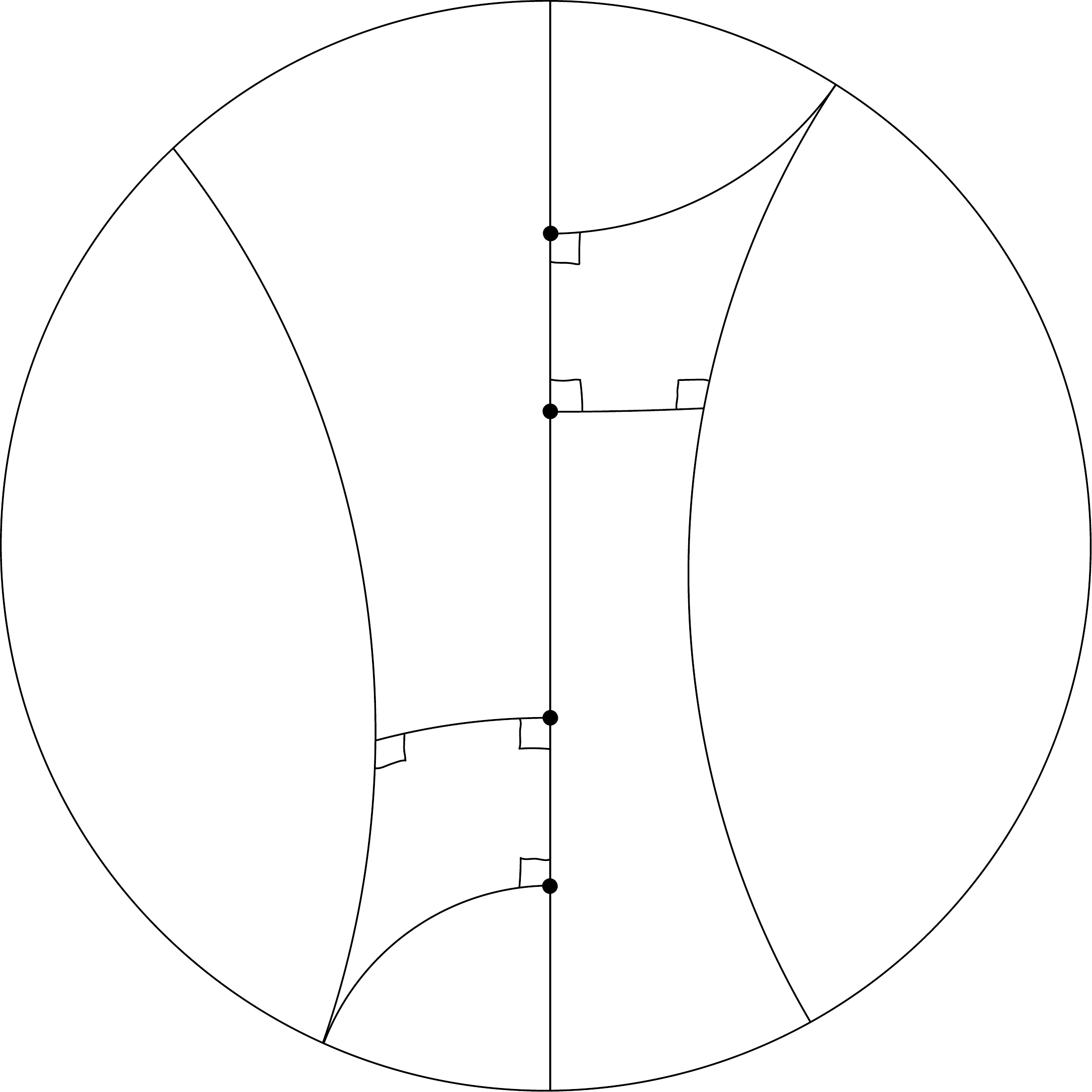,width=3.5in,height=3.5in,angle=0}}}
\vspace{-30pt}
\end{center}
\caption{$s(g_{4n+3})=\ell (PQ)+\frac{\ell_{2n+2}}{2}+ \ell (RS).$} 
\end{figure}
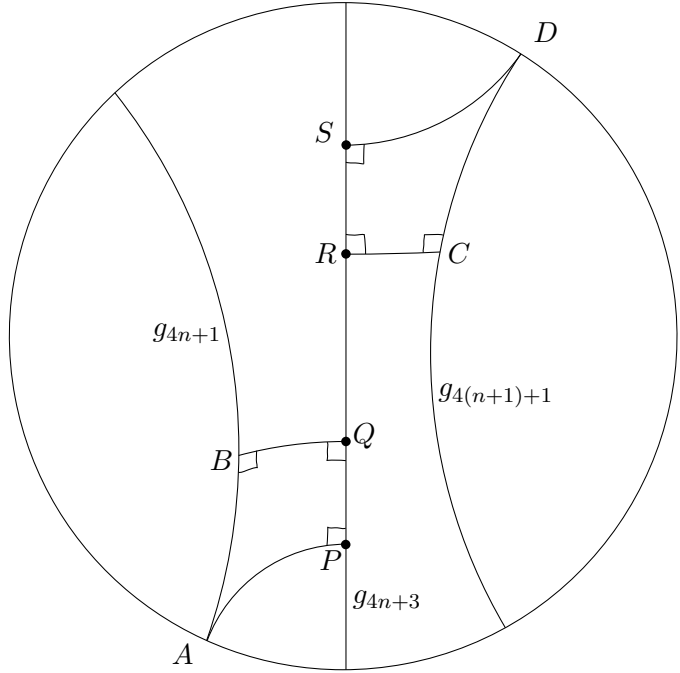

The starting choice of the half-twists and the fact that the index of $g_i$ has remainder $3$ under division by $4$ guarantees that the arc $QR$ is contained in arc $PS$ (see Figures 10 and 12). From Figure 12 we obtain 
$$s(g_{{4n+3}})=\ell (PQ)+\ell (RS)+\frac{1}{2}\ell_{2n+2},$$
where  $\ell (PQ)$ and $\ell (RS)$ are the lengths of arcs $PQ$ and $RS$, respectively. From the Lambert quadrilaterals $ABQP$ and $RSDC$, we obtain $\ell (PQ)=\sinh^{-1}\frac{1}{\sinh \ell (\eta_{2n+1})}$ and $\ell (RS)=\sinh^{-1}\frac{1}{\sinh \ell (\eta_{2n+2})}$ which gives the formula for $s(g_{4n+3})$.
\end{proof}

\vskip 1 cm

\noindent {\it Proof of Theorem 5.1.}  By Theorem \ref{thm:equiv-par-com}, it is enough to prove that $X=\{ (\ell_n,1/2)\}_n$ is complete. Since $X$ is symmetric with symmetry that is exchanging the front and the back side, it is enough to prove that the infinite polygon lift of $X^*$ to $\mathbb{D}$ has only one accumulation point on $S^1$, in addition to its vertices (see Corollary \ref{cor:complete-one-sided}). Therefore, it will be enough to prove that the sequence of nested geodesics $\{ g_n\}_{n=1}^{\infty}$ (from Figure 10) does not accumulate in $\mathbb{D}$.

It is immediate that $\sum_{n=1}^{\infty} \ell (\eta_n)=\infty$ implies that $X^*$ has only one point of accumulation on $S^1$ {in} addition to its vertices. Therefore we assume that $\sum_{n=1}^{\infty} \ell (\eta_n)<\infty$ in the rest of the proof. This implies that 
\begin{equation}
\label{eq:prod-finite}
1\leq \prod_{n=1}^{\infty}(1+\ell (\eta_n))<\infty.
\end{equation}

By Proposition A.1 in Appendix (or by the proof of \cite[Theorem C]{Saric2011}), the sequence $\{ g_n\}_{n=1}^{\infty}$ does not accumulate in $\mathbb{D}$ if and only if the piecewise horocyclic arc connecting the adjacent geodesics has infinite length. Denote by $s_n=s(g_n)$ the shear of $g_n$ with respect to the ideal quadrilateral whose  vertices are the ideal endpoints of $g_{n-1}$ and $g_{n+1}$ for $n\geq 2$. We do not define the shear of $g_1$. We start the piecewise horocyclic path on $g_1$ such that the part in the wedge between $g_1$ and $g_2$ has length $e^{-s_1}$. By Proposition A.3, the length of the part of the piecewise horocyclic path between $g_n$ and $g_{n+1}$ is
$$
e^{-s_1-s_2-\cdots -s_n}
$$
when $n$ is odd, and it equals
$$
e^{s_1+s_2+\cdots +s_n}
$$
when $n$ is even.

We will use the inequalities 
$$
e^{\sinh^{-1}\frac{1}{\sinh x}}>\frac{2}{x}
$$
and 
$$
\sinh x > x
$$
for $x>0$.

By Lemmas \ref{lem:shear-odd} and \ref{lem:shear-even}  and the above inequalities, we get, for $n\geq 1$,
\begin{equation}
\label{eq:est-shears}
\begin{array}l
e^{s_{4n+1}}> \frac{4}{\ell (\eta_{2n})\ell (\eta_{2n+1})} e^{-\frac{\ell_{2n+1}}{2}},\\
\\
e^{s_{4n+3}}> \frac{4}{\ell (\eta_{2n+1})\ell (\eta_{2n+2})} e^{\frac{\ell_{2n+2}}{2}},\ \mathrm{and}\\
\\
e^{s_{2n}}> \frac{[\ell (\eta_{n})]^2}{4}.
\end{array}
\end{equation}
Note that the constants $4$ are essential for what follows since they will be cancelled out which {will} facilitate the needed inequality comparison.

Since we need a lower estimate of the length of the piecewise horocyclic path, we note that $\ell (h)$ is greater than 
\begin{equation}
\label{eq:sum-4n}
\sum_{n=1}^{\infty}e^{s_{4n}+s_{4n-1}+\ldots +s_2+s_1}.
\end{equation}
The sum (\ref{eq:sum-4n}) can be written as
$$
\sum_{n=1}^{\infty}\prod_{k=0}^{n-1} e^{s_{4(k+1)}+s_{4k+3}+s_{4k+2}+s_{4k+1}}
$$
and by the {estimates in} (\ref{eq:est-shears}) we get
\begin{equation}
\begin{split}
\prod_{k=0}^{n-1} e^{s_{4(k+1)}+s_{4k+3}+s_{4k+2}+s_{4k+1}}> \frac{[\ell (\eta_{2k+2})]^2}{4} \cdot \frac{4}{\ell(\eta_{2k+1})
\ell (\eta_{2k+2})}\cdot e^{\frac{\ell_{2k+2}}{2}}\\ \cdot\frac{[\ell (\eta_{2k+1})]^2}{4} \cdot \frac{4}{\ell(\eta_{2k})\ell (\eta_{2k+1})} \cdot e^{\frac{-\ell_{2k+1}}{2}}=\frac{\ell (\eta_{2k+2})}{\ell (\eta_{2k})} e^{\frac{\ell_{{2k+2}}-\ell_{2k+1}}{2}}.
\end{split}
\end{equation}

This implies 
\begin{equation}
\label{eq:sum-prod-est}
\sum_{n=1}^{\infty}\prod_{k=0}^{n-1} e^{s_{4(k+1)}+s_{4k+3}+s_{4k+2}+s_{4k+1}}> 
C\sum_{n=1}^{\infty}\prod_{k=1}^{n-1} \frac{\ell (\eta_{2k+2})}{\ell (\eta_{2k})}e^{-\frac{\ell_{2k+1}}{2}+\frac{\ell_{2k+2}}{2}}.
\end{equation}

By cancellations we get
$$
\prod_{k=1}^{n-1} \frac{\ell (\eta_{2k+2})}{\ell (\eta_{2k})}e^{-\frac{\ell_{2k+1}}{2}+\frac{\ell_{2k+2}}{2}}= 
\frac{\ell (\eta_{2n})}{\ell (\eta_2)} e^{(\ell_{2n}-\ell_{2n-1}+\cdots {+\ell_4-\ell_3})/2}{.}
$$

By (\ref{eq:eta_n}), we have that $\ell (\eta_{2n})> e^{-\frac{\ell_{2n+1}}{2}}$, which together with (\ref{eq:sum-prod-est}) and the above equality gives for a modified constant $C$ that
\begin{equation}
\label{eq:sum-4n-est}
\sum_{n=1}^{\infty}e^{s_{4n}+s_{4n-1}+\ldots +s_2+s_1}> C\sum_{n=1}^{\infty} e^{-\frac{\sigma_{2n+1}}{2}}.
\end{equation}
   
By Lemmas \ref{lem:shear-odd} and \ref{lem:shear-even} we obtain
\begin{equation}
\label{eq:sum-4n+1-est}
\begin{split}
e^{-s_{4n+1}-s_{4n}-\ldots -s_2}= \big{[}e^{-\sinh^{-1}\frac{1}{\sinh \ell (\eta_{2n})}}\cdot e^{-\sinh^{-1}\frac{1}{\sinh \ell (\eta_{2n+1})}}\cdot e^{\frac{\ell_{2n+1}}{2}}\big{]}\cdot \\\big{[}\frac{1}{\sinh^2\frac{\ell (\eta_{2n})}{2}} \big{]}\cdot
 \big{[}e^{-\sinh^{-1}\frac{1}{\sinh \ell (\eta_{2n-1})}}\cdot e^{-\sinh^{-1}\frac{1}{\sinh \ell (\eta_{2n})}}\cdot e^{{-}\frac{\ell_{2n}}{2}}\big{]} \cdot\\ \big{[}\frac{1}{\sinh^2\frac{\ell (\eta_{2n-1})}{2}} \big{]}\cdot\ldots \cdot \big{[}\frac{1}{\sinh^2\frac{\ell (\eta_{1})}{2}} \big{]}
\end{split}
\end{equation}

By using the inequalities $e^{-\sinh^{-1}\frac{1}{\sinh x}}> \frac{x}{5}$ and $\frac{e^{-\sinh^{-1}\frac{1}{\sinh x}}}{\sinh \frac{x}{2}}>\frac{1}{1+x}$ for small $x>0$, we conclude that the right hand side of (\ref{eq:sum-4n+1-est}) is greater than
$$
\big{[}\prod_{i=1}^{2n}\frac{1}{1+\ell (\eta_i)}\big{]}^2 e^{-\sinh^{-1}\frac{1}{\sinh \ell (\eta_{2n+1})}} e^{\frac{\ell_{2n+1}-\ell_{2n}+\cdot +\ell_3-\ell_2}{2}}.
$$
By the above inequalities and by (\ref{eq:prod-finite}) we have {for another constant $C>0$} that
\begin{equation}
\label{eq:sum-4n+1-est-final}
\sum_{n=1}^{\infty} e^{-s_{4n+1}-s_{4n}-\ldots -s_2-s_1}> C\sum_{n=1}^{\infty} e^{-\frac{\sigma_{2n+2}}{2}}.
\end{equation}

By summing (\ref{eq:sum-4n-est}) and (\ref{eq:sum-4n+1-est-final}) we obtain {for some constant $C>0$} that the piecewise horocyclic path has length $\ell (h)$ greater than
$$
C\sum_{n=1}^{\infty} e^{-\frac{\sigma_{n}}{2}}
$$
and the assumption of the theorem implies that it is of infinite length. Thus $\tilde{X}^{*}$ accumulates to exactly one point in addition to its vertices. This implies that $X$ {is parabolic}. {\it End of the proof of Theorem 5.1.}

\subsection{A slice of the space of flutes}
Let $X_{a,b}$, with $a>0$ and $b>0$, and
$$
\ell_{2n}=a\ln (n+1)+b\ln n ,\ \ \ell_{2n+1}=(a+b)\ln (n+1)
$$
be a half-twist flute surface with the above lengths of geodesics on the boundary of a pants decomposition (see \cite{BHS}). It is immediate that $\ell_n$ is an increasing sequence. By Theorem \ref{thm:half-twist-O_G} we have that $X_{a,b}$ {is parabolic} if and only if $\sum_{n=1}^{\infty}e^{-\frac{\sigma_n}{2}}=\infty$. By computations in \cite[Example 9.9]{BHS}, we get that $\sum_{n=1}^{\infty}e^{-\frac{\sigma_n}{2}}\asymp \sum_{k=1}^{\infty}k^{-\frac{\min (a,b)}{2}}$. Then we conclude that in Figure 1 the two domains $X_{a,b}=?$ {consist entirely of parabolic flutes}.

\begin{figure}[h]
\leavevmode \SetLabels
\endSetLabels
\begin{center}
\AffixLabels{\centerline{\epsfig{file =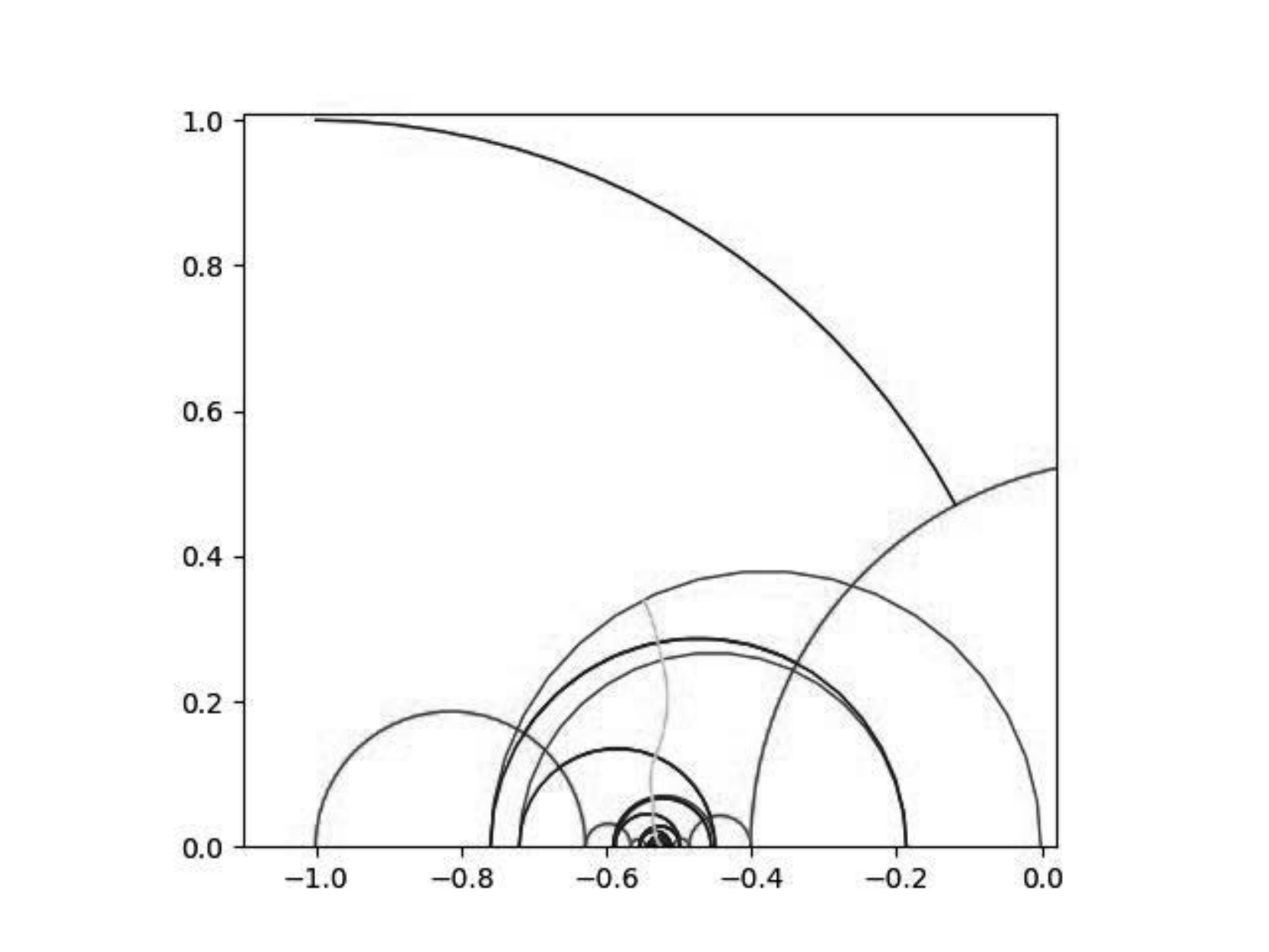,width=4in,height=2.5in,angle=0}}}
\vspace{-30pt}
\end{center}
\caption{A computer generated path of horocyclic concatenations for the flute surface $X_{a,b}$ with $a=4$ and $b=1$.} 
\end{figure}

\subsection{The symmetric finite ends surfaces} Let $X$ be a Riemann surface with finitely many ends accumulated by genus as in Section \ref{sec:finite-ends-par}.
Denote by $\{ X_i\}_{i=1}^k$ the end surfaces. For each $X_i$ assume that the twists on the geodesics $\{\alpha_n\}_{n=1}^{\infty}$ accumulating to the end is $1/2$ as in Figure 7. Denote by $\ell_n$ the lengths of simple closed geodesics $\alpha_n$ accumulating  to the single end. Let $\beta_n$ be the simple closed geodesics that cuts off a torus with one hole and let $\gamma_n$ be a simple closed geodesic in the torus as in Figure 7. 

\begin{theorem}
\label{thm:symmetric-1/2-finite-ends}
Let $X$ be a hyperbolic surface with finitely many  end surfaces $\{ X_i\}_{i=1}^k$ accumulated by genus as in Figure 7. Assume that the lengths of the simple closed geodesics $\beta_n$ and $\gamma_n$ are between two positive constants for each end surface. Let $\ell_n$ be the lengths of simple closed geodesics $\alpha_n$ accumulating at each end. Assume that $\ell_n$ is an increasing sequence and the twists on $\alpha_n$ are all equal to $1/2$.
{Then}, for each $X_i$,
$$
\sum_{n=1}^{\infty}e^{-\frac{\sigma_n}{2}}=\infty
$$
if and only if $X$ {is parabolic}, where $\sigma_n=\ell_n-\ell_{n-1}+\cdots +(-1)^{n-1}\ell_1$.
\end{theorem}

\begin{proof}
Assume $\sum_{n=1}^{\infty}e^{-\frac{\sigma_n}{2}}=\infty$ for any $X_i$. By using quasiconformal maps, we can assume that all the cut off tori are isometric to each other (i.e. the geodesics $\beta_n$ and $\gamma_n$ have the same length) and the twists on $\beta_n$ and $\gamma_n$ are equal to $0$. Then there exists a front to back decomposition and we denote by $X^*_i$ the front side of the end surface $X_i$. Let $\widetilde{\alpha_n}$ be the lifts of $\alpha_n$ that connect two boundary sides of a single lift of $X^*\setminus \cup_n\gamma_n$. Then {the} $\widetilde{\alpha_n}$ are nested and by Remark \ref{rem:nested-alpha-acc} it is enough to prove that {the} $\widetilde{\alpha_n}$ accumulate to a single point on $S^1$. We adopt  the computation from the proof of Theorem \ref{thm:half-twist-O_G}.

\begin{figure}[h]
\leavevmode \SetLabels
\endSetLabels
\begin{center}
\AffixLabels{\centerline{\epsfig{file =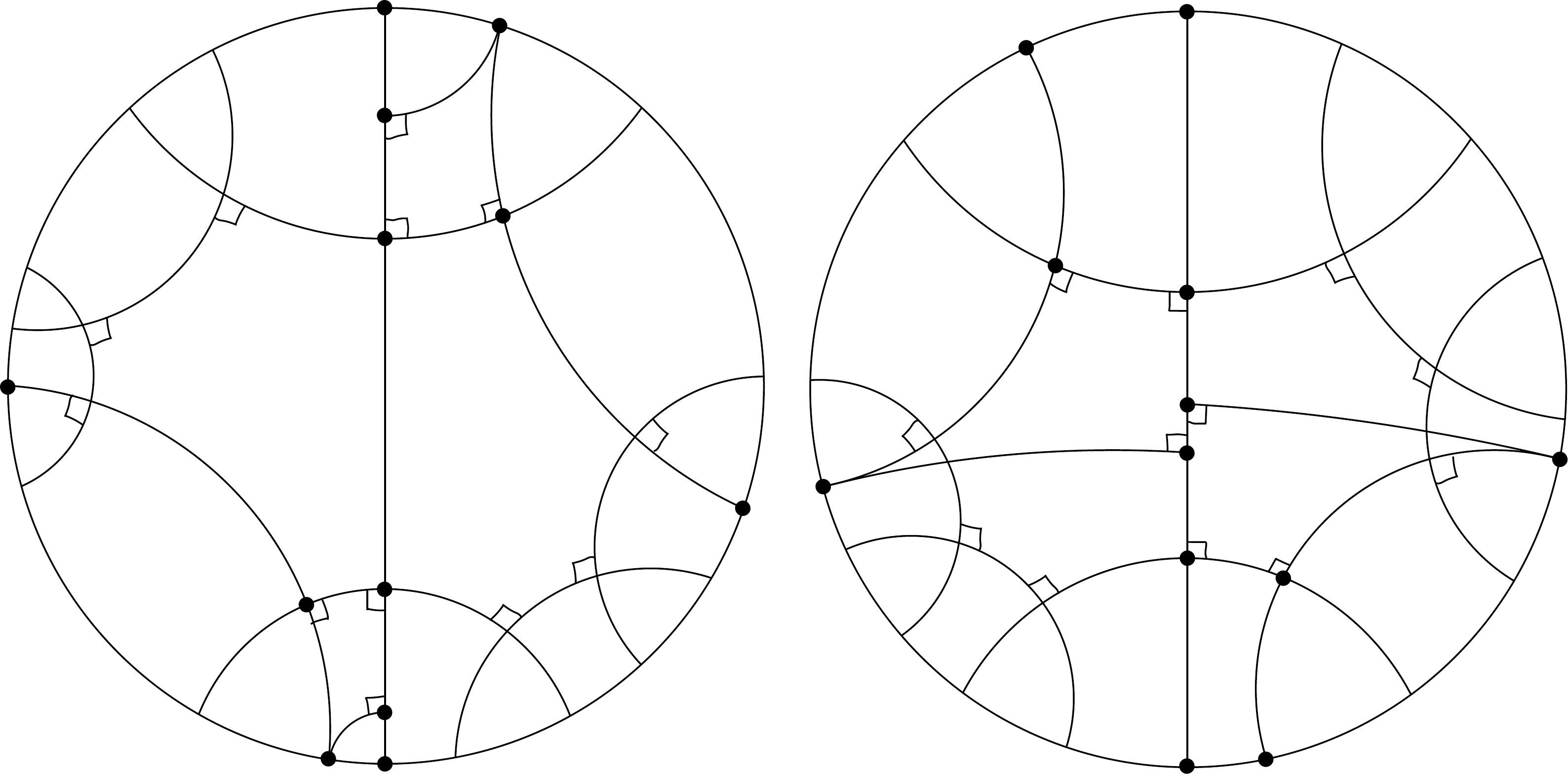,width=3.5in,height=1.8in,angle=0}}}
\vspace{-30pt}
\end{center}
\caption{Lift of half of pairs of pants of Loch-Ness monster and associated shears.} 
\end{figure}

Let $g_{2n-1}$ be the {geodesic that contains the lift} $\widetilde{\alpha_n}$ for $n=1,2,\ldots$ and let $g_{2n}$ be the added geodesic which {shares} one endpoint with $g_{2n-1}$ and {its} other endpoint with $g_{2n+1}$ as in Figure  10. Let $\eta_n$ be the orthogeodesic arc between the geodesics $g_{2n-1}$ and $g_{2n+1}$. Then $\eta_n$ is a side of a right angled hexagon that is on $X^*_i$ whose adjacent sides have lengths $\frac{\ell_n}{2}$ and $\frac{\ell_{n+1}}{2}$ and lie on $g_{2n-1}$ and $g_{2n+1}$, respectively. The side opposite $\eta_n$ is a lift of half of $\beta_n$ and therefore has length $\frac{\ell (\beta_n)}{2}$ (see Figure 14). By the hexagon formula and the assumption that $\ell_n$ is increasing and going to infinity, we get that
$$
e^{-\frac{\ell_{n+1}}{2}} < \ell (\eta_n) < C e^{-\frac{\ell_n}{2}}
$$
for some $C>0$ and all $n$.

The proof that $\sum_{n=1}^{\infty}e^{s_{4n}+s_{4n-1}+\ldots +s_2+s_1}> C\sum_{n=1}^{\infty} e^{-\frac{\sigma_{2n+1}}{2}}$ follows by the same lines as in the proof of Theorem \ref{thm:half-twist-O_G} because the geometric positions of the geodesics determining the shears on $g_n$ are identical as seen in {Figure 14}. The above estimate  $e^{-\frac{\ell_{n+1}}{2}}<\ell (\eta_n)$ finishes the proof of the inequality. 

The proof that $\sum_{n=1}^{\infty} e^{-s_{4n+1}-s_{4n}-\ldots -s_2-s_1}> C\sum_{n=1}^{\infty} e^{-\frac{\sigma_{2n+2}}{2}}$ also follows by the proof of Theorem \ref{thm:half-twist-O_G} and the geometric {positions} in Figure 14.

 Assume $\sum_{n=1}^{\infty}e^{-\frac{\sigma_n}{2}}<\infty$ for some $X_i$. Similar to what is done in [7, pgs 41-42] for half-twist flutes, we form a concatenation $p$ of the summits of Saccheri quadrilaterals with bases $\eta_n$ starting from the point on $\alpha_1$ furthest from $\eta_1$ that escapes the end $X_i$. Then the length of the n-th summit is at most $Ce^{-\frac{\sigma_{n-1}}{2}}$ for some positive constant $C$. The path $p$ is finite from our assumption and the fact that $X$ is not parabolic immediately follows.
\end{proof}

Consider the same Riemann surface $X$ with finitely many ends accumulated by genus except for each $X_i$ the twists on the geodesics $\{\alpha_n\}_{n=1}^{\infty}$ accumulating to its end are $0$ as in Figure 6.

\begin{theorem}
\label{thm:symmetric-0-finite-ends}
Let $X$ be a hyperbolic surface with finitely many  end surfaces $\{ X_i\}_{i=1}^k$ accumulated by genus as in Figure 6. Assume that the lengths of the simple closed geodesics $\beta_n$ and $\gamma_n$ are between two positive constants for each end surface. Let $\ell_n$ be the lengths of simple closed geodesics $\alpha_n$ accumulating at each end. Assume that $\ell_n$ is an increasing sequence and the twists on $\alpha_n$ are all equal to $0$.
Then, for each $X_i$,
$$
\sum_{n=1}^{\infty}e^{-\frac{l_n}{2}}=\infty
$$
if and only if $X$ is parabolic.
\end{theorem}

\begin{proof}
When $\sum_{n=1}^{\infty}e^{-\frac{l_n}{2}} =\infty$ in $X_i$, from the fact that $e^{-\frac{l_{n+1}}{2}} < l(\eta_{n})$  for all $n$, we get $\sum_{n=1}^{\infty}l(\eta_{n}) = \infty$. Our goal is to show that every escaping geodesic ray in $X_i$ is infinite length which implies $X$ is complete (and thus parabolic).

Call the front half of the end $X_i^{\star}$. Reflect an escaping geodesic ray $r$ on $X_i$ to its front. The result is a piecewise geodesic arc $r^{\star}$ in $X_i^{\star}$ with the same length as $r$.\\
\indent There are two cases to consider. The first case is that $r^{\star}$ enters finitely many (including possibly zero) attached toruses in $X_i$. In this case, $r^{\star}$ eventually enters no toruses. Notice that after this point the length of $r^{\star}$ between $\alpha_n$ and $\alpha_{n+1}$ is at least the length of $\eta_n$. The length of $r^{\star}$ then must be at least the length of a path along the $\eta_n$ starting from some sufficiently large index. Since $\sum_{n=k}^{\infty}l(\eta_{n}) =\infty$, the length of $r^{\star}$ is infinite.

The second case is that $r^{\star}$ enters infinitely many attached toruses. Since the lengths of $\beta_n$ and $\gamma_n$ are between two positive constants, we can assume without loss of generality that all of the attached toruses are isomorphic. That means the collar widths around each of the $\beta_n$ are the same. Since $r^{\star}$ passes through the full width of infinitely many identical collars, the length of $r^{\star}$ is infinite.

Assume $\sum_{n=1}^{\infty}e^{-\frac{l_n}{2}} < \infty$ and remember that $l(\eta_n) < Ce^{-\frac{l_n}{2}}$ for all $n$ for some positive constant $C$. Then the path along the $\eta_n$ on the surface is finite. Thus, the covering group is of the second kind and $X$ is not parabolic.
\end{proof}

\section*{Appendix}

Let $\{g_n\}_{n=1}^{\infty}$ be a sequence of nested geodesics in $\mathbb{D}$ such that any two adjacent geodesics $g_n$ and $g_{n+1}$ share an ideal endpoint, and no three geodesics share an ideal endpoint. The space between $g_n$ and $g_{n+1}$ is called a {\it wedge} and the common endpoint of $g_n$ and $g_{n+1}$ is called the {\it vertex} of the wedge. Each wedge is foliated by horocyclic arcs orthogonal to the sides that lie on horocycles whose center is the vertex of the corresponding wedge.

Fix a point $P_1$ on $g_1$. There exists a unique piecewise horocylic path $h$ starting at $P_1$ that consists of horocyclic arcs connecting the sides of the wedges of $\{ g_n\}_{n=1}^{\infty}$. In this section we will establish that the set of nested geodesics accumulates to a single point on the unit circle if and only if the path $h$ has infinite length.
We also prove the formulas for the lengths of the horocyclic arcs in the wedges of the geodesics $\{ g_n\}_{n=1}^{\infty}$  in terms of the shears on $\{ g_n\}$. Both results use the ideas from \cite[Theorem C]{Saric2011} and we give a proof for the convenience of the reader.

\vskip .2 cm

\noindent {\bf Proposition A.1.}
{\it Let $\{g_n\}_{n=1}^{\infty}$ be a nested sequence of geodesics and $h$ a piecewise horocyclic path as above. Then the nested sequence $\{ g_n\}_n$ accumulates to a single point on $S^1$ if and only if $h$ has infinite length.}

\vskip .2 cm

\begin{proof}
Assume that $h$ has a finite length. If $g_n$ {accumulates} to a single point on $S^1$ then the path $h$ has a sequence of points that converges toward that point on $S^1$. Since the distance from $P_1$ to any point on the boundary is infinite, it follows that $h$ has an infinite length which is a contradiction. Thus, $g_n$ {does} not accumulate to a single point on $S^1$. 

Conversely, assume that $g_n$  {accumulates} to a geodesic $g^*$ in $\mathbb{D}$. We need to prove that $h$ has a finite length. Let $a$ be the orthogeodesic arc between $g_1$ and $g^*$. The length of $a$ is finite and we will compare the length of $h$ to the length of $a$ (see Figure 15).

 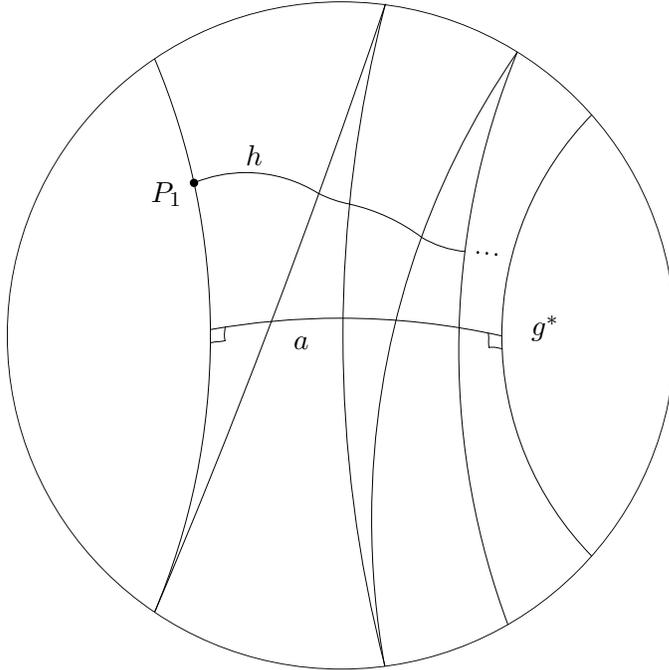
\begin{figure}[h]
\leavevmode \SetLabels
\L(.3*.7) $P_1$\\
\L(.4*.755) $h$\\
\L(.45*.48) $a$\\
\L(.7*.5) $g^*$\\
\endSetLabels
\begin{center}
\AffixLabels{\centerline{\epsfig{file =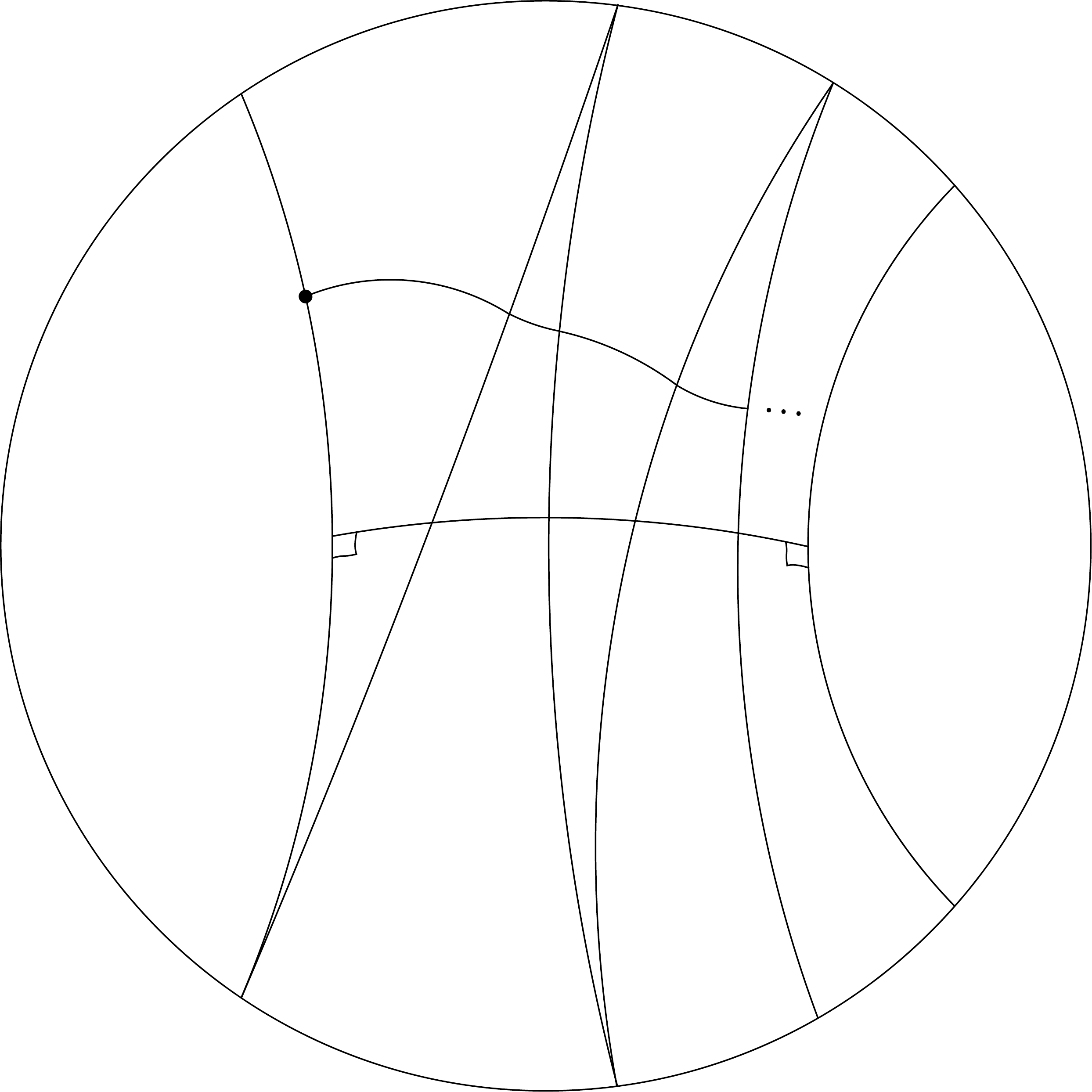,width=3.5in,height=3.5in,angle=0}}}
\vspace{-30pt}
\end{center}
\caption{The piecewise horocyclic arc $h$ and the orthogeodesic $a$.} 
\end{figure}

Denote by $W_n$ the wedge between $g_n$ and $g_{n+1}$. Let $a_n=a\cap W_n$ and $h_n=h\cap W_n$. Let $d_n$ be the geodesic arc on $g_n$ whose endpoints are $a\cap g_n$ and $h\cap g_n$. 

 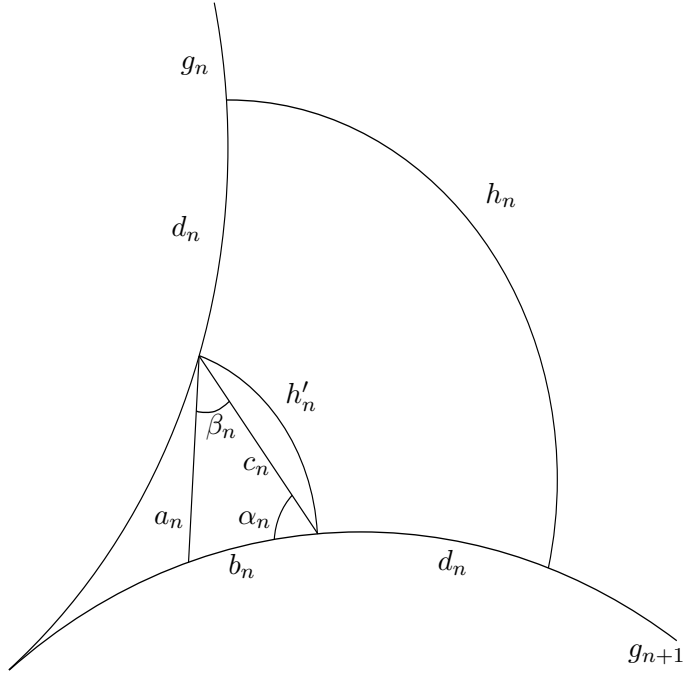
\begin{figure}[h]
\leavevmode \SetLabels
\L(.33*.9) $g_n$\\
\L(.65*.7) $h_n$\\
\L(.8*.02) $g_{n+1}$\\
\L(.44*.4) $h_n'$\\
\L(.6*.15) $d_n$\\
\L(.32*.65) $d_n$\\
\L(.38*.145) $b_n$\\
\L(.302*.22) $a_n$\\
\L(.395*.3) $c_n$\\
\L(.39*.22) $\alpha_n$\\
\L(.356*.354) $\beta_n$\\
\endSetLabels
\begin{center}
\AffixLabels{\centerline{\epsfig{file =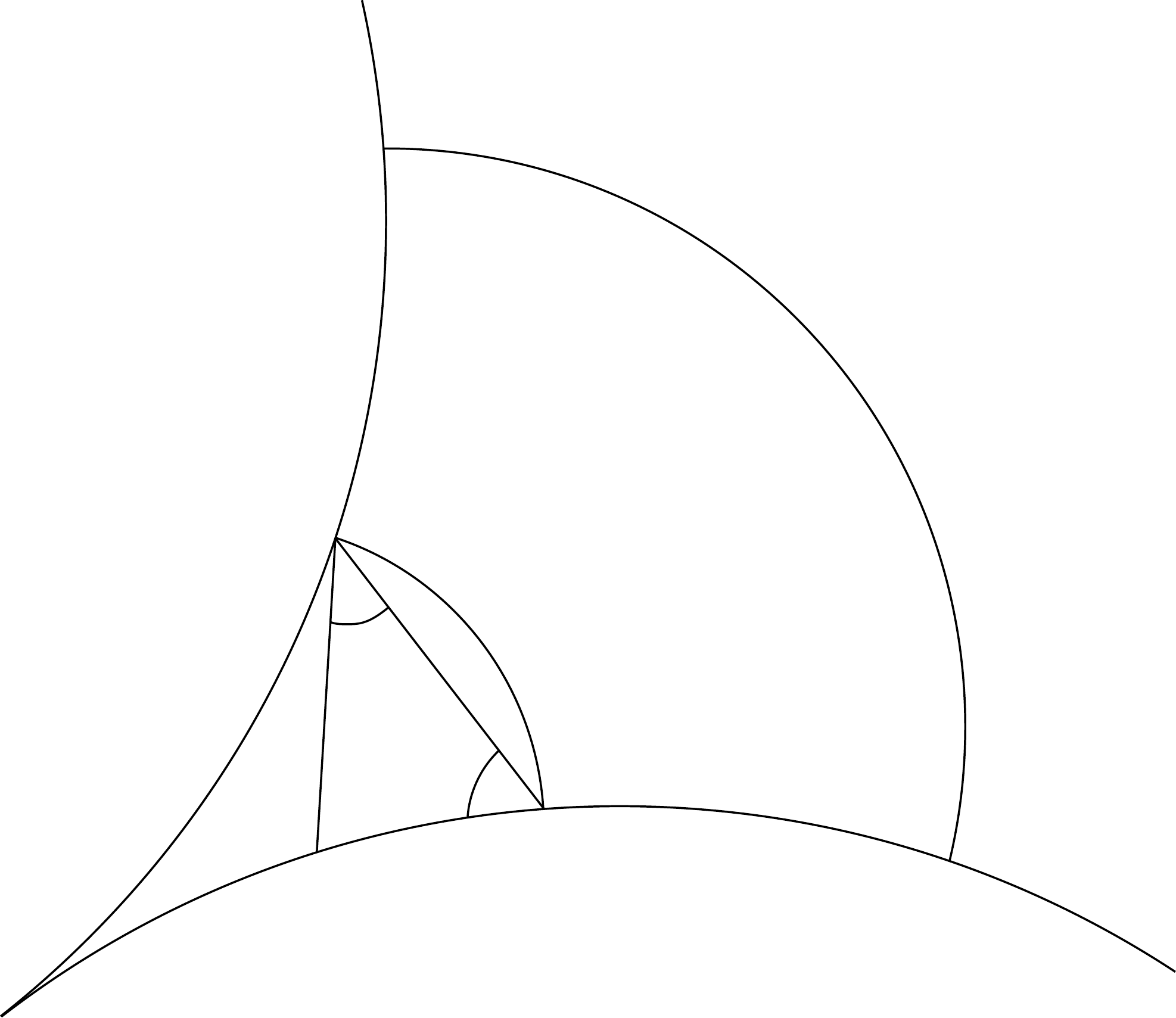,width=3.5in,height=3.5in,angle=0}}}
\vspace{-30pt}
\end{center}
\caption{The comparison between $a_n$ and $h_n$.} 
\end{figure}

Let $h_n'$ be the horocyclic arc in $W_n$ which starts at $a_n\cap g_n$. Let $b_n$ be the geodesic arc on $g_{n+1}$ between $a\cap g_{n+1}$ and $h_n'\cap g_{n+1}$. Let $c_n$ be the geodesic arc that connects the endpoints of $h_n'$. Let $\alpha_n$ be the angle facing $a_n$ and let $\beta_n$ be the angle facing $b_n$ in the geodesic triangle with sides $a_n$, $b_n$ and $c_n$ (see Figure 16). By the sine formula, we have
\begin{equation}
\label{thm:sine-law}
\frac{\sinh \ell(a_n)}{\sin\alpha_n}=\frac{\sinh \ell (b_n)}{\sin\beta_n}.
\end{equation}

We prove that $\alpha_n\in [\alpha_l,\alpha_u]$ for all $n$, where $0<\alpha_l<\alpha_u<\pi$. To do so, map the wedge $W_n\subset\mathbb{D}$ by a M\"obius map to the wedge in $\mathbb{H}$ with vertex $\infty$ whose boundary geodesics are $(0,\infty )$ and $(1,\infty )$ are images of $g_n$ and $g_{n+1}$, respectively. The images of $a_n$, $b_n$, $c_n$, $h_n'$  and $\alpha_n$ are denoted by the same letters (see Figure {16}). Since the length $\ell (a_n)$ is bounded above by $\ell (a)<\infty$ for all $n$, it follows that there is a positive lower bound $y_0>0$ on the heights of the points in $a_n$ for each $n$. The horocyclic arc $h_n'$ is on the Euclidean height at least $y_0>0$ and the Euclidean circle that contains the geodesic arc $c_n$ has center $1/2\in\mathbb{R}$ and radius greater than $y_0$. By elementary Euclidean geometry, these bounds imply the existence $0<\alpha_l<\alpha_u<\pi$ such  that $\alpha_n\in [\alpha_l,\alpha_u]$ for all $n$.

 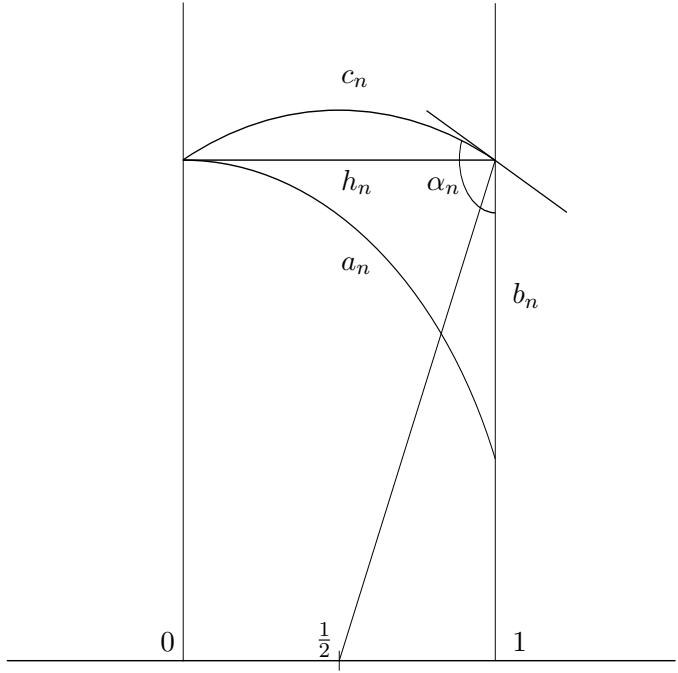
\begin{figure}[h]
\leavevmode \SetLabels
\L(.5*.88) $c_n$\\
\L(.5*.72) $h_n$\\
\L(.5*.6) $a_n$\\
\L(.31*.03) $0$\\
\L(.473*.035) $\frac{1}{2}$\\
\L(.68*.03) $1$\\
\L(.59*.72) $\alpha_n$\\
\L(.68*.55) $b_n$\\
\endSetLabels
\begin{center}
\AffixLabels{\centerline{\epsfig{file =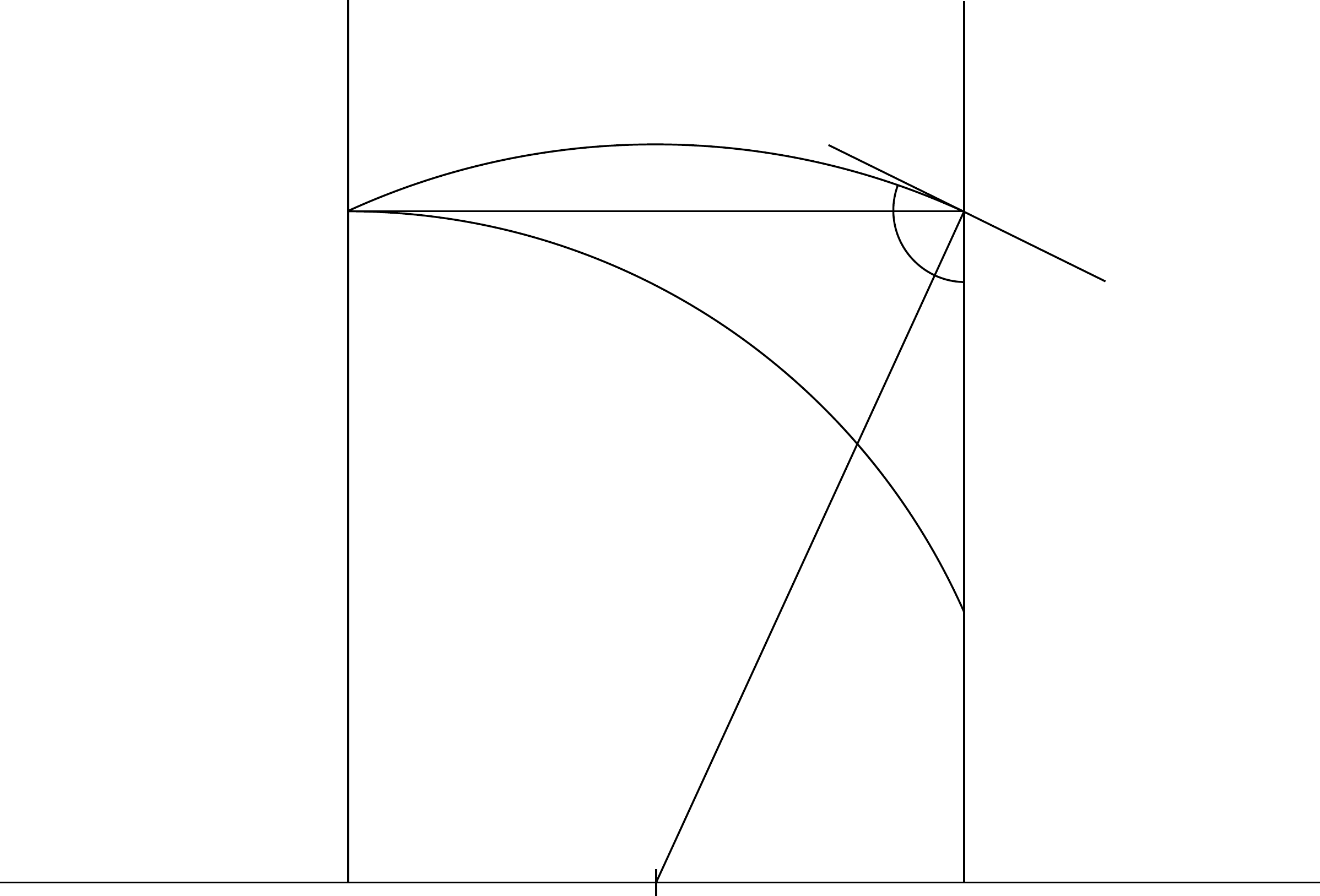,width=3.5in,height=3.5in,angle=0}}}
\vspace{-30pt}
\end{center}
\caption{The wedge with vertex $\infty$.} 
\end{figure}

By (\ref{thm:sine-law}) and $\alpha_n\in [\alpha_l,\alpha_u]$, we conclude that, for all $n$,
\begin{equation}
\label{eq:b_n-est}
\ell (b_n)\leq C\ell (a_n)
\end{equation}
for some constant $C$. By Figure 16 and equation (\ref{eq:b_n-est}), we have
$$
\ell (d_{n+1})=\ell (d_n)+\ell (b_n)\leq  \ell (d_n)+C\ell (a_n)
$$
which implies
\begin{equation}
\label{eq:dist-a_n-h_n}
\ell (d_{n+1})\leq \ell (d_1)+C\sum_{i=1}^n \ell (a_i)\leq \ell (d_1)+C\ell (a).
\end{equation}

Given quantities $A$ and $B$, the notation $A\asymp B$ means that there exists a constant $k>0$ such that $1/k\leq A/B\leq k$. 
By (\ref{eq:dist-a_n-h_n}), we have $\ell (h_n')\asymp\ell (h_n)$. By 
$\ell (c_n)\leq \ell (a_n)+\ell (b_n)\leq (1+C)\ell (a_n)\leq (1+C)\ell (a)$ and by the fact that $h_n'$ and $c_n$ share endpoints, we conclude that $\ell (h_n')\asymp\ell (c_n)$. Therefore
$$
\ell (h)=\sum_{n=1}^{\infty} \ell (h_n)\asymp\sum_{n=1}^{\infty}\ell (h_n')\asymp (1+C)\sum_{n=1}^{\infty}\ell (a_n)=(1+C)\ell (a)<\infty
$$
and the proposition is proved.
\end{proof}

We compute the length of the horocyclic arc $h_n$ in terms of the shears on the geodesics $\{ g_n\}_{n=1}^{\infty}$. 
Consider two wedges $W_1$ and $W_2$ that share a common boundary geodesic $g_2$ such that their vertices are the opposite ideal endpoints of $g_2$. Let $g_1$ and $g_3$ be the other boundary geodesics of $W_1$ and $W_2$, respectively. 
Orient $g_2$ such that $g_1$ is on its left. 

If $g_2$ shares the initial point with $g_1$ (and thus the terminal point with $g_3$) we say that the pair of wedges $(W_1,W_2)$ is {\it left-open} (see left side of Figure {18}). If $g_2$ shares the initial point with $g_3$ (and thus the terminal point with {$g_1$}) we say that the pair of wedges $(W_1,W_2)$ is {\it left-closed} (see right side of Figure {18}).

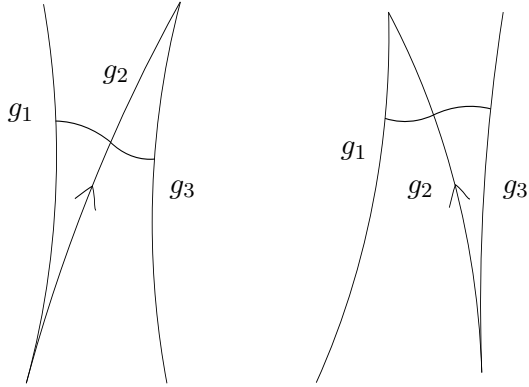
\begin{figure}[h]
\leavevmode \SetLabels
\L(.23*.7) $g_1$\\
\L(.33*.8) $g_2$\\
\L(.4*.5) $g_3$\\
\L(.58*.6) $g_1$\\
\L(.65*.5) $g_2$\\
\L(.75*.5) $g_3$\\
\endSetLabels
\begin{center}
\AffixLabels{\centerline{\epsfig{file =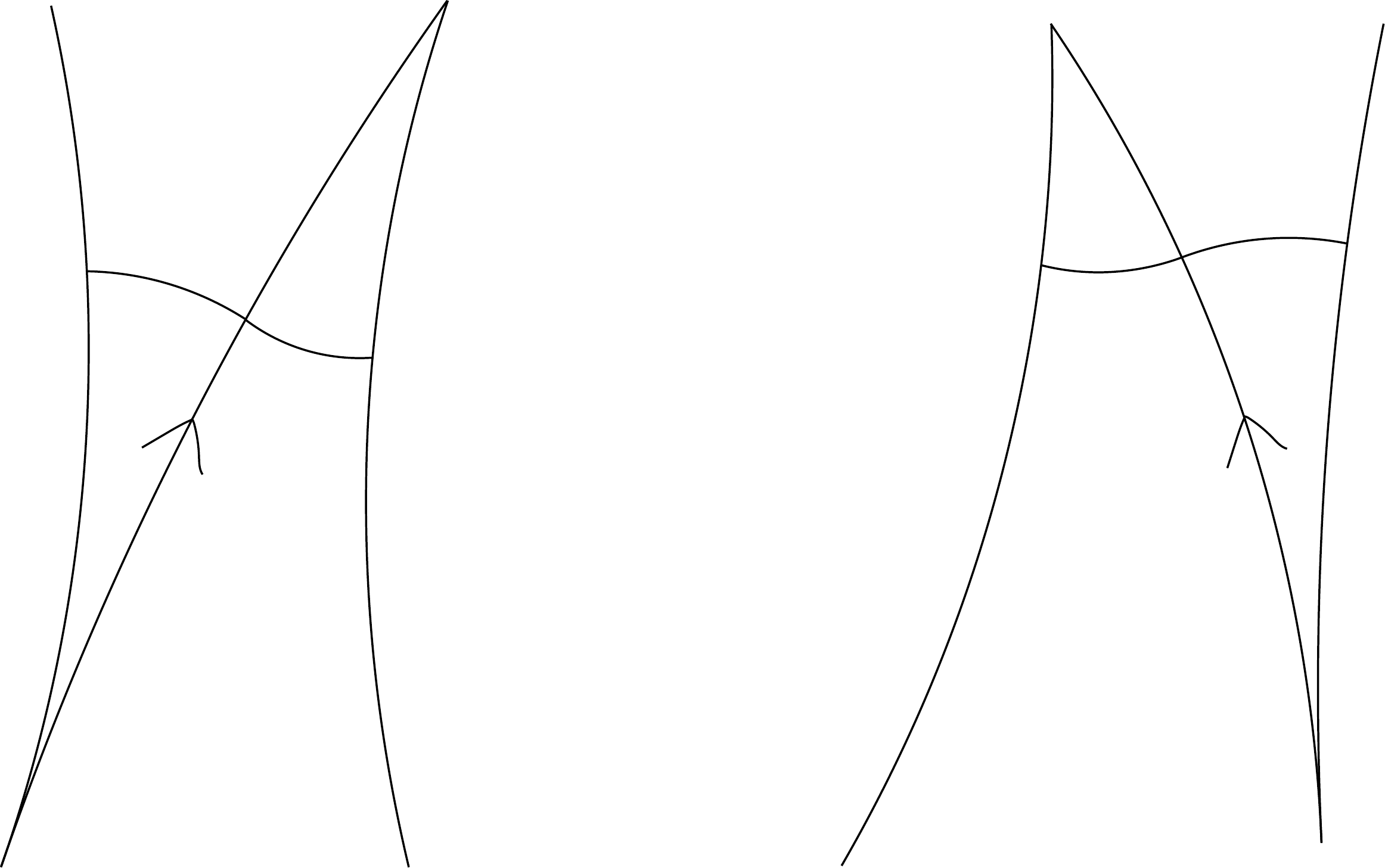,width=2.5in,height=2in,angle=0}}}
\vspace{-30pt}
\end{center}
\caption{The left-open and the left-closed pair of wedges.} 
\end{figure}

\vskip .2 cm

\noindent {\bf Lemma A.2.} {\it Let $(W_1,W_2)$ be a pair of adjacent wedges and let $h_1$ be a horocyclic arc orthogonal to and connecting the boundary sides of $W_1$. Let $s(g_2)$ be the shear along the common boundary geodesic $g_2$ of  
$(W_1,W_2)$ of the ideal quadrilateral with vertices equal to the endpoints of $g_1$ and $g_3$. Let $h_2$ be the horocyclic arc in $W_2$ orthogonal to its boundary that continues $h_1$.

If $(W_1,W_2)$ is left-open then
$$
\ell (h_2)=\frac{e^{s(g_2)}}{\ell (h_1)},
$$
where $\ell (h_1)$ and $\ell (h_2)$ are the lengths of $h_1$ and $h_2$.

If $(W_1,W_2)$ is left-closed then
$$
\ell (h_2)=\frac{e^{-s(g_2)}}{\ell (h_1)}.
$$}

\begin{proof}
Assume that $(W_1,W_2)$ is left-closed. Map $(W_1,W_2)$ by a M\"obius map into $\mathbb{H}$ such that the common geodesic is $g_2=(0,\infty )$ and the other vertex of $g_1$ is $-1$. Then necessarily we have $g_3=(0, e^{s(g_2)})$. The horocyclic arc $h_1$ is a horizontal Euclidean arc between $g_1$ and $g_2$ that meets $g_2$ (i.e. the $y$-axis) at a point $iy_1$ (see Figure 19).

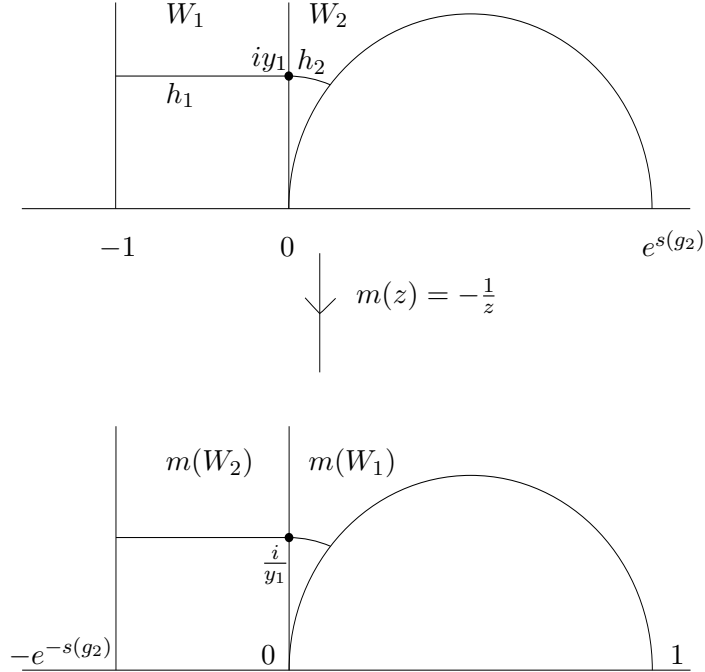
\begin{figure}[h]
\leavevmode \SetLabels
\L(.23*.62) $-1$\\
\L(.42*.62) $0$\\
\L(.8*.62) $e^{s(g_2)}$\\
\L(.3*.97) $W_1$\\
\L(.45*.97) $W_2$\\
\L(.3*.85) $h_1$\\
\L(.438*.9) $h_2$\\
\L(.39*.905) $iy_1$\\
\L(.3*.3) $m(W_2)$\\
\L(.4*.15) $\frac{i}{y_1}$\\
\L(.45*.3) $m(W_1)$\\
\L(.135*.01) $-e^{-s(g_2)}$\\
\L(.83*.01) $1$\\
\L(.4*.01) $0$\\
\L(.5*.55) $m(z)=-\frac{1}{z}$\\
\endSetLabels
\begin{center}
\AffixLabels{\centerline{\epsfig{file =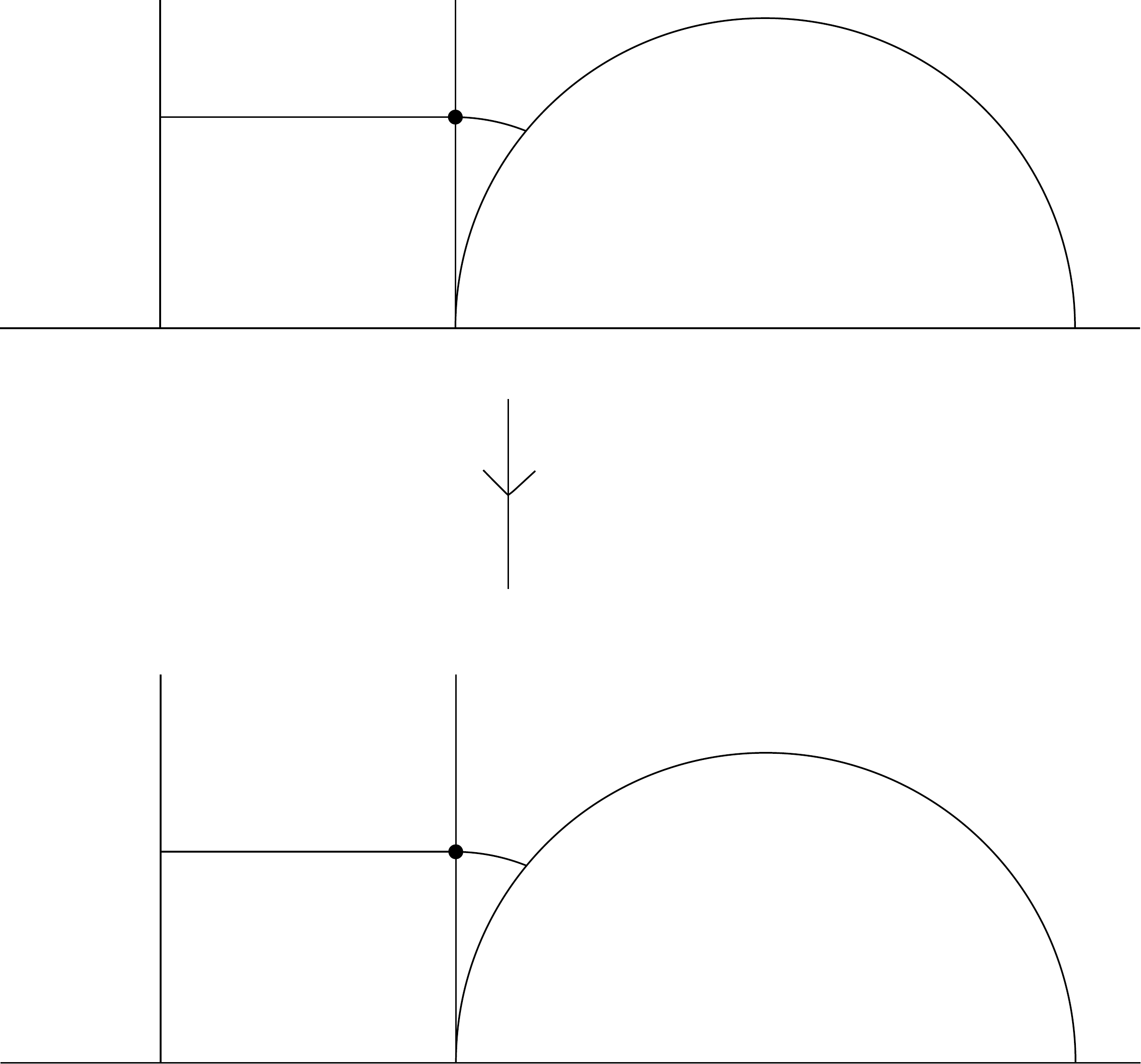,width=3.5in,height=3.5in,angle=0}}}
\vspace{-30pt}
\end{center}
\caption{The left-closed pair of wedges with the vertex $\infty$.} 
\end{figure}

We have that $\ell (h_1)=\frac{1}{y_1}$. Let $m(z)=-\frac{1}{z}$. Then $m(g_3)=(-e^{-s(g_2)},\infty )$ and $m(g_2)=g_2$. The horocyclic arc $m(h_2)$ is the horizontal Euclidean arc between $m(g_3)$ and $g_2$ with Euclidean height $\frac{1}{y_1}$. 
Therefore
$$
\ell (h_2)=\ell (m(h_2))=y_1e^{-s(g_2)}=\frac{e^{-s(g_2)}}{\ell (h_1)}.
$$

 The case when $(W_1,W_2)$ is left-open is dealt with in an analogous fashion. 
\end{proof}

We find the length of the $n$-th horocyclic arc $h_n$ based on the length of $h_1$ and the shears on the geodesics before $h_n$.
 
 \vskip .2 cm
 
\noindent {\bf Proposition A.3.} {\it
Let $\{ g_n\}_{n=1}^{\infty}$ be the above nested family of geodesics and let $s_n=s(g_n)$ be the corresponding shears. Let $h$ be a curve obtained by concatenation of horocyclic arcs $h_n$ orthogonal to and connecting two boundary geodesics of each wedge $W_n$ starting at a point $P_1\in g_1$. Choose $P_1\in g_1$ such that the horocyclic arc $h_1=h\cap W_1$ has length $e^{-s_1}$. Then, for $n$ odd,
$$
\ell (h_n)=e^{-s_1-s_2-\cdots -s_n}
$$
and, for $n$ even,
$$
\ell (h_n)=e^{s_1+s_2+\cdots +s_n}
$$
}

\begin{proof}
Consider the wedge $W_n$ for $n>1$. By Lemma A.2, we get  $\ell (h_2)=\frac{e^{s_2}}{\ell (h_1)}=e^{s_1+s_2}$. Applying Lemma A.2 again, we get $\ell (h_3)=\frac{e^{-s_3}}{\ell (h_2)}=\frac{e^{-s_3}}{e^{s_1+s_2}}=e^{-(s_1+s_2+s_3)}$, and so on. The result follows by induction.
\end{proof}

\begin{figure}[h]
\leavevmode \SetLabels
\endSetLabels
\begin{center}
\AffixLabels{\centerline{\epsfig{file =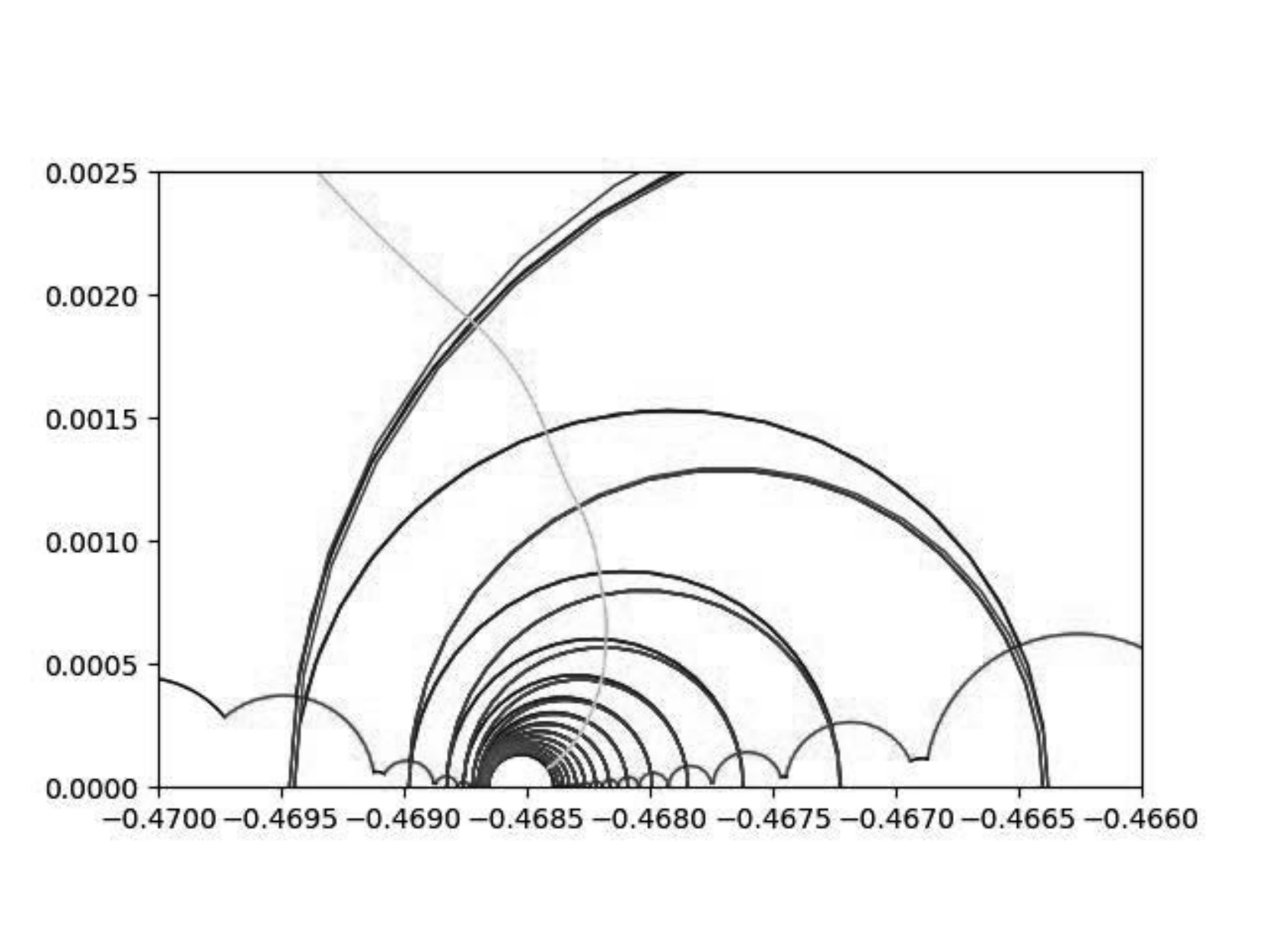,width=4in,height=2.5in,angle=0}}}
\vspace{-30pt}
\end{center}
\caption{A computer generated path of horocyclic concatenations for the Loch-Ness monster surface $X_{a,b}$ with $a=.5$ and $b=6.2$.} 
\end{figure}


\vskip .2 cm

\begin{thebibliography}{Thua}

\vskip .5cm

\bibitem{Agard} Agard, S. \textit{A geometric proof of Mostow's rigidity theorem for groups of divergence type.}
Acta Math. 151 (1983), no. 3-4, 231-252.

\bibitem{AhlforsSario} L. Ahlfors and L. Sario, {\it Riemann surfaces}, 
Princeton Mathematical Series, No. 26 Princeton University Press, Princeton, N.J. 1960.
 
 
\bibitem{ALPS} D. Alessandrini, L. Liu, A. Papadopoulos and W. Su, {\it On the inclusion of the quasiconformal Teichm\"uller space into the length-spectrum Teichm\"uller space}, Monatsh. Math. 179 (2016), no. 2, 165-189. 

\bibitem{AlvarezRodriguez} V. \' Alvarez and J. Rodr\' iguez, {\it Structure theorems for Riemann and topological surfaces}, J. London Math. Soc. (2) 69 (2004), no. 1, 153-168.

\bibitem{Astala-Zinsmeister}
Astala, K.; Zinsmeister, M. {\it{ Mostow rigidity and Fuchsian groups.} }
C. R. Acad. Sci. Paris S\'er. I Math. 311 (1990), no. 6, 301-306.

\bibitem{Basmajian} A. Basmajian, {\it Hyperbolic structures for surfaces of infinite type}, Trans. Amer. Math. Soc. 336 (1993), no. 1, 421-444.

\bibitem{BHS} A. Basmajian, H Hakobyan and D. \v Sari\' c, {\it The type problem for Riemann surfaces via Fenchel-Nielsen parameters}, arXiv:2011.03166.

\bibitem{Bishop} Bishop, C., \textit{Divergence groups have the Bowen property}, Ann. of Math. (2) 154 (2001), no. 1, 205-217.

\bibitem{BasmajianSaric} A. Basmajian and D. \v Sari\' c, {\it Geodesically complete hyperbolic structures}, Math. Proc. Cambridge Philos. Soc. 166 (2019), no. 2, 219-242.

\bibitem{Fernandez-Melian}  Fern\'andez, J. L.; Meli\'an, M. V., {\it Escaping geodesics of Riemannian surfaces.} Acta Math. 187 (2001), no. 2, 213-236.


\bibitem{KahnMarkovic} J. Kahn and V. Markovic, {\it Immersing almost geodesic surfaces in a closed hyperbolic three manifold}, Ann. of Math. (2) 175 (2012), no. 3, 1127-1190. 

\bibitem{Ker} B. Ker\' ekj\' art\' o, {\it Vorlesungen \" uber Topologie. I}, Springer, Berlin, 1923. 

\bibitem{Kinjo} E. Kinjo, {\it On Teichm\"uller metric and the length spectrums of topologically infinite Riemann surfaces}, Kodai Math. J. 34 (2011), no. 2, 179-190. 

\bibitem{McMullen} C. McMullen, {\it Hausdorff dimension and conformal dynamics. III. Computation of dimension}, Amer. J. Math. 120 (1998), no. 4, 691-721. 

\bibitem{Milnor} J. Milnor, {\it On deciding whether a surface is parabolic or hyperbolic}, 
Amer. Math. Monthly 84 (1977), no. 1, 43-46. 

\bibitem{Nevanlinna:criterion} Nevanlinna, R. \textit{\"Uber die Existenz von beschr\"ankten Potentialfunktionen auf Fl\"achen von unendlichem Geschlecht.} (German)
Math. Z. 52 (1950), 599-604.


\bibitem{Nicholls} P. Nicholls, {\it Fundamental regions and the type problem for a Riemann surface}, Math. Z.  174  (1980), no. 2, 187-196.

\bibitem{Nicholls1} P. Nicholls, {\it The ergodic theory of discrete groups}, London Mathematical Society Lecture Note Series, 143. Cambridge University Press, Cambridge, 1989.


\bibitem{Richards} I. Richards, {\it On the classification of noncompact surfaces}, Trans. Amer. Math. Soc. 106 (1963), 259-269.

\bibitem{Saric2011} D. \v Sari\' c, {\it Circle homeomorphisms and shears}, Geom. Topol. 14 (2010), no. 4, 2405-2430.

\bibitem{Saric} D. \v Sari\' c, {\it Quadratic differentials and foliations on infinite Riemann surfaces}, preprint, arXiv:2207.08626.

\bibitem{Shiga} H. Shiga, {\it On a distance defined by the length spectrum of Teichm\"uller space}, Ann. Acad. Sci. Fenn. Math. 28 (2003), no. 2, 315-326. 

\bibitem{Sullivan} D. Sullivan, {\it On the ergodic theory at infinity of an arbitrary discrete group of hyperbolic motions.} Riemann surfaces and related topics: Proceedings of the 1978 Stony Brook Conference (State Univ. New York, Stony Brook, N.Y., 1978), pp. 465-496, Ann. of Math. Stud., 97, Princeton Univ. Press, Princeton, N.J., 1981. 

\bibitem{Tukia} Tukia, P. \textit{Differentiability and rigidity of M\"obius groups.}
Invent. Math. 82 (1985), no. 3, 557-578.

\bibitem{Portilla} A. Portilla, J. M. Rodr\' iguez and E. Tour\' is, {\it Structure theorem for Riemannian surfaces with arbitrary curvature}, Math. Z. 271, 45–62 (2012).

\bibitem{Garnett} J. B. Garnett and D. E. Marshall, {\it Harmonic measure}. Cambridge University Press, July 7, 2008, 132-133.



\end{thebibliography}
\end{document}